\documentclass{amsart}
\usepackage{amsmath}
\usepackage{amssymb}
\usepackage{amsthm,amsxtra}
\usepackage{cite}
\usepackage{enumerate}
\usepackage{enumitem}
\usepackage[mathscr]{eucal}
\usepackage{float}
\usepackage{centernot}
\usepackage{adjustbox}
\newtheorem{Def}{Definition}[section]
\newtheorem{Th}{Theorem}[section]
\newtheorem{Ex}{Example}[section]

\newtheorem{Lemma}{Lemma}[section]
\newtheorem{Prop}{Proposition}[section]
\newtheorem{Cor}{Corollary}[section]
\newtheorem{Rem}{Remark}[section]

\DeclareMathOperator{\Int}{{Int}}

\DeclareMathOperator{\non}{{non}}
\begin{document}

\title[On localization of the Menger property]{On localization of the Menger property}

\author[D. Chandra, N. Alam]{Debraj Chandra$^*$, Nur Alam$^*$}
\newcommand{\acr}{\newline\indent}
\address{\llap{*\,}Department of Mathematics, University of Gour Banga, Malda-732103, West Bengal, India}
\email{debrajchandra1986@gmail.com, nurrejwana@gmail.com}

\thanks{The second author
is thankful to University Grants Commission (UGC), New Delhi-110002, India for granting UGC-NET Junior Research Fellowship (1173/(CSIR-UGC NET JUNE 2017)) during the tenure of which this work was done.}


\maketitle

\begin{abstract}
In this paper we introduce and study the local version of the Menger property, namely locally Menger property (or, locally Menger space). We explore some preservation like properties in this space. We also discuss certain situations where this local property behaves somewhat differently from the classical Menger property. Some observations about the character of a point, network weight and weight in this space are also investigated carefully. We also introduce the notion of Menger generated space (in short, MG-space) and make certain investigations in this space. Several topological observations on the decomposition and the remainder of locally Menger spaces are also discussed.
\end{abstract}
\smallskip
\noindent{\bf\keywordsname{}:} {Menger property, locally Menger, MG-space, decomposition, remainder, nearly perfect mapping, meshing}

\section{Introduction}
A methodical lesson of selection principles was started by Scheepers \cite{coc1} (see also \cite{coc2}). Since then the field has flourished into one of the most active research areas of set theoretic topology. Subsequently several topological properties have been defined and characterized in terms of selection principles. Studies of selection principles and their interrelationship have a long renowned history. Readers interested in selection principles and their current works can consult the papers \cite{survey,survey1,WMP,ARWR,DCPD,BPFS,WMB,MSPC,
PRSP,LjSM} where more references can be found. %Several variations of the classical selection principles were introduced and studied by various authors (see \cite{sM-II,variations}).
%Also very recently Ko\v{c}inac \cite{set-Menger} introduced set-selection principles and investigate their relationships with earlier defined covering properties.
One of the important selection principle is the Menger property (see \cite{coc1,coc2}). The study on the localization of the classical selection principles is not well known. The purpose of this paper is to work on the localization of the Menger property.
%We introduce some new classes of topological spaces, namely locally Scheepers and $SG$-spaces. In the process we obtain several important topological observations such as separation properties, preservation like properties, and also study their interrelationship. It is also observed that some of these properties do not hold in general for Scheepers spaces. Some observations involving quotient mappings are also investigated in locally Scheepers as well as $SG$-spaces. Under CH (Continuum Hypothesis), we discuss about the character of a point, network weight and weight in the context of locally Scheepers $P$-spaces. The behaviours of locally Scheepers property in regular spaces are also investigated carefully.

This paper is organised as follows. In Section 3, we introduce a more general class of topological spaces, namely locally Menger spaces. We obtain some equivalent formulations of this space under regularity conditions. Also we make certain observations in $P$-spaces. It is shown that a locally Menger Hausdorff $P$-space can be densely embedded in a Menger Hausdorff space. We are able to provide certain insight into the character of a point, network weight and weight in the context of locally Menger $P$-spaces.
%Later in this section, some separation type of properties are also discussed.
In Section 4, we present preservation like properties under topological operations in locally Menger spaces. It is observed that some of these properties do not hold in general for Menger spaces. Later in this section, we introduce the notion of Menger generated spaces (in short, MG-spaces) and show that every such space can be obtained as a quotient of some locally Menger space. We also make several investigations in these spaces. In the remaining part of this section, we study the decomposition \cite{Willard} of locally Menger spaces. The class $\mathfrak{M}$ of all MG-spaces can be shown to be identical with the class of all spaces which are obtained as a decomposition of locally Menger spaces. The final section is devoted to study the remainder
of locally Menger spaces. We also define nearly perfect mapping (a closed continuous surjective mapping between two topological spaces is called nearly perfect if every fiber is Menger) and obtain few interesting observations.
%Most of the results of this section are obtained as a subsequent observation in the realm of paracompact $p$-spaces \cite{p-space,EGT}.

\section{Preliminaries}
Throughout the paper $(X,\tau)$ stands for a topological space. For undefined notations and terminologies, see \cite{Engelking}.

A space $X$ is said to be Menger if for each sequence $(\mathcal{U}_n)$ of open covers of $X$ there is a sequence $(\mathcal{V}_n)$ such that for each $n\in \mathbb{N}$, $\mathcal{V}_n$ is a finite subset of $\mathcal{U}_n$ and $\cup_{n\in\mathbb{N}}(\cup \mathcal{V}_n)=X$ \cite{coc1}.

 It is to be noted that every $\sigma$-compact space is Menger and every Menger space is Lindel\"{o}f, and the Menger property is hereditary for $F_\sigma$ subsets and continuous images. Also the Menger property is preserved under countable unions \cite{coc2}.

%We say that a subset $U$ is a neighbourhood of $x$ in $X$ if $U$ is an open set containing $x$.
For a subset $A$ of a space $X$, $\overline{A}^X$ denotes the closure of $A$ in $X$. If no confusion arises, $\overline{A}^X$ can be denoted by $\overline{A}$.
A subset $L$ of a space $X$ is said to be locally closed at a point $x\in L$ if there is a neighbourhood $V$ of $x$ in $X$ such that $L\cap V$ is a closed subset of the subspace $V$. $L$ is said to be locally closed in $X$ if it is locally closed at each $x\in L$. %\cite{Bourbaki}

The weight $w(X)$ of $X$ is the smallest possible cardinality of a base for $X$ and the character $\chi(x,X)$ of a point $x$ in  $X$ is the smallest cardinality of a local base for $x$. A family $\mathcal{N}$ of subsets of  $X$ is said to be a network for $X$ if for each $x\in X$ and any neighbourhood $U$ of $x$ there exists a $A\in\mathcal{N}$ such that $x\in A\subseteq U$. The network weight $nw(X)$ of  $X$ is defined as the smallest cardinal number of the form $|\mathcal{N}|$, where $\mathcal{N}$ is a network for $X$. Clearly $nw(X)\leq w(X)$ and $nw(X)\leq |X|$. It is also clear that if $\mathcal{N}$ is a network of  $X$ such that $|\mathcal{N}|\leq\kappa$, then $X$ has a dense subset of cardinality less than or equal to $\kappa$.

Also, for a collection $\mathcal{C}$ of subsets of $X$, let $\mathcal{C}_\delta$ denote the collection of all sets that can be expressed as the intersection of some nonempty subcollection of $\mathcal{C}$ and let $\mathcal{C}_{\delta,\sigma}$ denote the collection of all sets that can be expressed as the union of some subcollection of $\mathcal{C}_\delta$. We say that $\mathcal{C}$ is a source for a space $Y$ in $X$ if $Y$ is a subspace of $X$ such that $Y\in\mathcal{C}_{\delta,\sigma}$. A source $\mathcal{C}$ for $Y$ in $X$ is called open (closed) if every member of $\mathcal{C}$ is  open (respectively, closed) in $X$ and a source $\mathcal{C}$ is countable if $\mathcal{C}$ is countable.

Let $\mathbb{N}^\mathbb{N}$ be the Baire space. A natural pre-order $\leq^*$ on $\mathbb{N}^\mathbb{N}$ is defined by $f\leq^*g$ if and only if $f(n)\leq g(n)$ for all but finitely many $n$. A subset $D$ of $\mathbb{N}^\mathbb{N}$ is said to be dominating if for each $g\in\mathbb{N}^\mathbb{N}$ there exists a $f\in D$ such that $g\leq^* f$. It can be observed that any Menger subspace of $\mathbb{N}^\mathbb{N}$ is not dominating \cite{coc2}.
%Note that any non-empty open subset $U$ of $\mathbb{N}^\mathbb{N}$ is dominating.
Let $\mathfrak{d}$ be the minimum cardinality of a dominating subset of $\mathbb{N}^\mathbb{N}$ and $\mathfrak{c}$ be the cardinality of the set of reals. It is well known that $\aleph_0<\mathfrak{d}\leq\mathfrak{c}$.

Recall that a family $\mathcal{A}\subseteq P(\mathbb{N})$ is said to be an almost disjoint family if each $A\in\mathcal{A}$ is infinite and for any two distinct elements $B,C\in\mathcal{A}$, $|B\cap C|<\aleph_0$. For an almost disjoint family $\mathcal{A}$, let $\Psi(\mathcal{A})=\mathcal{A}\cup\mathbb{N}$ be the Isbell-Mr\'{o}wka space. It is well known that $\Psi(\mathcal{A})$ is a locally compact zero-dimensional Hausdorff space \cite{Mrowka}.

\section{Certain investigations on the locally Menger spaces}
\subsection{The locally Menger property}
We first introduce the main definition of the paper.
\begin{Def}
 A space $X$ is said to have the locally Menger property if for each $x\in X$ there exist an open set $U$ and a Menger subspace $Y$ of $X$ such that $x\in U\subseteq Y$.
\end{Def}
We also call a space $X$ is locally Menger if $X$ has the locally Menger property.
It is immediate that every Menger space is locally Menger and for a Lindel\"{o}f space, Menger and locally Menger properties are equivalent.

Recall that a space $X$ is said to be locally $\sigma$-compact (respectively, locally Lindel\"{o}f) if for each $x\in X$ there exist an open set $U$ and a $\sigma$-compact (respectively, Lindel\"{o}f) subspace $Y$ of $X$ such that $x\in U\subseteq Y$. Clearly every locally $\sigma$-compact space is locally Menger and every locally Menger space is locally Lindel\"{o}f.

%\begin{figure}[h]
%\begin{adjustbox}{max width=\textwidth,max height=\textheight,keepaspectratio,center}
%\begin{tikzcd}[column sep=4ex,row sep=4ex,arrows={crossing over}]
%%level 5
%&& \text{Lindel\"{o}f} \arrow[rr]&&
%\text{locally Lindel\"{o}f}
%\\
%%level 4
%& \text{Menger} \arrow[ur]\arrow[rr]&&
%\text{locally Menger}\arrow[ur]&
%\\
%%level 3
%\text{Hurewicz} \arrow[ur]\arrow[rr]&&
%\text{locally Hurewicz} \arrow[ur]&&
%\\
%%level 2
%&\text{Rothberger}\arrow[uu]\arrow[rr]&&
%\text{locally Rothberger}\arrow[uu]&
%\\
%%level 1 (from bottom)
%\text{$\sigma$-compact} \arrow[uu]\arrow[rr]&&
%\text{locally $\sigma$-compact}\arrow[uu]&&
%\\
%%\level 1'
%\text{compact} \arrow[u]\arrow[rr]&&
%\text{locally compact}\arrow[u]&&
%\end{tikzcd}
%\end{adjustbox}
% \caption{Diagram of local properties}
% \label{dig1}
%\end{figure}

The following example shows that there exists a Tychonoff locally Menger space which is not Menger.
\begin{Ex}
\label{E1}
Let $\mathcal{A}\subseteq P(\mathbb{N})$ be an almost disjoint family of cardinality $\mathfrak{c}$ and $\Psi(\mathcal{A})$ be the Isbell-Mr\'{o}wka space. $\Psi(\mathcal{A})$ is locally Menger as it is locally compact. Observe that $\mathcal A$ is not Lindel\"of as it is a discrete subspace of cardinality $\mathfrak c$. Since the Menger property is preserved under closed subspaces, it follows that $\Psi(\mathcal A)$ is not Menger.
\end{Ex}

By \cite[Theorem 5.1]{coc2}, there exists a separable metric space $X$ with $|X|=\omega_1$ which is Menger but not $\sigma$-compact. It is easy to observe that every locally $\sigma$-compact Lindel\"{o}f space is $\sigma$-compact. Clearly such $X$ is locally Menger but not locally $\sigma$-compact.

Let $\mathbb{P}$ be the space of irrationals. It is well known that $\mathbb{P}$ is Lindel\"{o}f but not Menger. Hence $\mathbb{P}$ is locally Lindel\"{o}f but not locally Menger.

We now observe the following characterization of the locally Menger property in regular spaces.
\begin{Th}
\label{T7}
Let $X$ be a regular space. The following assertions are equivalent.
\begin{enumerate}[wide,label={\upshape(\arabic*)},
leftmargin=*]
\item $X$ is locally Menger.

\item For each $x\in X$ and for each open set $V$ with $x\in V$ there exist an open set $U$ and a Menger subspace $Y$ of $X$ such that $x\in U\subseteq Y\subseteq V$.

\item For each $x\in X$ there exists an open set $U$ with $x\in U$ such that $\overline{U}$ is Menger.

\item $X$ has a basis consisting of closed Menger neighbourhoods.\\
\end{enumerate}
\end{Th}
%\begin{proof}
%We only give proof of $(1)\Rightarrow (2).$ Let $x\in X$ and $V$ be an open set containing $x$. Choose an open set $W$ such that $x\in W \subseteq\overline{ W }\subseteq V$. Also we can find an open set $U$ and a Menger subspace $Y$ of $X$ such that $x\in U\subseteq Y$. Clearly $ W \cap U$ is an open set containing $x$ and $\overline{ W }\cap Y$ is the required Menger subspace contained in $V$.
%%\\
%%\noindent$(2)\Rightarrow(3).$ Let $x\in X$. Then there exist an open set $U$ and a Menger subspace $Y$ of $X$ such that $x\in U\subseteq Y\subseteq X$. Also there exists an open set $ Z $ such that $x\in Z \subseteq\overline{ Z }\subseteq U$. Clearly $\overline{ Z }$ is the required  Menger subspace in this case.\\
%%
%%\noindent$(3)\Rightarrow(4).$ Let $x\in X$ and  $V$ be an open set containing $x$. Choose an open set $U$ and a Menger subspace $Y$ of $X$ such that $x\in U\subseteq Y\subseteq V$. Also we can find an open set $ Z $ such that $x\in Z \subseteq\overline{ Z }\subseteq U$. Clearly $\overline{ Z }$ is Menger and hence $X$ has a basis consisting of closed Menger neighbourhoods.\\
%%
%%\noindent$(4)\Rightarrow (2).$ Let $x\in X$ and $V$ be an open set containing $x$. We can find an open set $U$ and a closed Menger set $Y$ such that $x\in U\subseteq Y$. Also there is an open set $ Z $ such that $x\in Z \subseteq\overline{ Z }\subseteq U\cap V$. Clearly $\overline{ Z }$ is Menger and $(2)$ is satisfied.
%\end{proof}

By Example~\ref{E1}, it can be easily concluded that none of the conditions of Theorem~\ref{T7} implies that $X$ is Menger.
%\begin{Ex}
% \label{E3}
% Let $X=[0,\omega_1)\times I^I$, where $[0,\omega_1)$ is the set of all countable ordinals with the interval topology and $I^I$ (where $I=[0,1]$) is with the cartesian product topology. Now $X$ is a locally compact regular space (see \cite{Lynn}) and hence $X$ is locally Menger. Clearly $X$ is not Menger as it is not Lindel\"{o}f.
% \end{Ex}

It is important to note that the regularity condition is essential in Theorem~\ref{T7}.
\begin{Ex}\hfill
%\mbox{}
\begin{enumerate}[wide=0pt,label={\upshape\bf(\arabic*)},
ref={\theEx (\arabic*)}]
\item \label{EN02} Let $X=\mathbb{N}^\mathbb{N}$ be the Baire space and $aX$ be the one point compactification of $X$. It is easy to observe that $aX$ is Menger which is not regular.
We now claim that $X$ is not locally Menger. Assume the contrary. If possible let $X$ be locally Menger. We can choose a nonempty open set $U$ and a Menger subspace $Y$ such that $U\subseteq Y$. Since every nonempty open subset of the Baire space is dominating, it follows that $Y$ is dominating. Which contradicts the fact that any Menger subspace of $X$ is not dominating. An easy verification shows that $aX$ satisfies the conditions $(1)$ and $(3)$. It can also be shown that $aX$ does not satisfy the conditions $(2)$ and $(4)$ as $X$ is not locally Menger.\\
\item \label{E14} Let $X$ be the real line $\mathbb R$ with the particular point $p$ topology \cite{Lynn}. Clearly $X$ is neither regular nor Menger. Since every nonempty open set in $X$ is dense in $X$, $X$ does not satisfy the conditions $(3)$ and $(4)$. It can be easily shown that $X$ satisfies the other two conditions.
    \end{enumerate}
\end{Ex}

\begin{Rem}
We remark that none of the conditions of Theorem~\ref{T7} implies that $X$ is regular.
 Let $X$ be the set of all irrational numbers in $(0,\infty)$ and $A$ be the set of irrational numbers in $(0,1)$. For each $x\in A$, let $A_x=\{n+x : n=0,1,2,\dotsc\}$ and choose $Y=\{A_x : x\in A\}$. Consider the topology $\tau=\{\cup B \; :\; B\subseteq Y\}$ on $X$. Clearly $X$ is not regular as it is not Hausdorff. Observe that $X$ satisfies condition $(2)$. Let $x\in X$ and choose $V\in\tau$ such that $x\in V$. Write $x=k+y$ for some non negative integer $k$ and some $y\in A$. Clearly $A_y$ is an open set contained in $V$ and also $A_y$ is Menger as it is compact. Thus $X$ satisfies condition $(2)$ and also $(1)$ of Theorem~\ref{T7}.

 To show that condition $(4)$ is satisfied, let $x\in X$ and $V$ be an open subset of $X$ containing $x$. Then for some $y\in A$ there exists an open subset $A_y$ of $X$ containing $x$ such that $A_y$ is Menger and  $A_y\subseteq V$. It is clear that $\overline{A_y}=A_y$ and the condition $(4)$ is satisfied. It also follows that $X$ satisfies condition $(3)$.
%\end{Ex}
\end{Rem}

Recall that the radial interval topology on $\mathbb{R}^2$ is generated by the basis consisting all open intervals disjoint from the origin which lie on the lines passing through the origin, together with the sets of the form $\cup\{E_\theta : 0\leq\theta<\pi\}$, where each $E_\theta$ is a nonempty  open interval centered at the origin on the line of slope $\tan \theta$ \cite{Lynn}.

In the following example we show that a regular space is not always locally Menger (not even locally Lindel\"{o}f).
\begin{Ex}
\label{P10}
 Let $X=\mathbb{R}^2$ with the radial interval topology. Clearly $X$ is regular. We now observe that $X$ is not locally Menger. It is enough to show that $X$ is not locally Lindel\"{o}f. Assume the contrary. Choose an open set $U$ containing the origin and a Lindel\"{o}f subspace $Y$  such that $U\subseteq Y$. By regularity we can find an open set $V$ containing the origin such that $\overline{V}\subseteq U$. Clearly $\overline{V}$ is Lindel\"{o}f. Now $V=\cup_{0\leq\theta<\pi}E_\theta$, for some nonempty open interval $E_\theta$ centered at origin on the line of slope $\tan\theta$. For convenience choose $E_\theta =(x_\theta,y_\theta)$ and let $F_\theta$ be the middle half of $E_\theta$. Choose $G=\cup_{0\leq\theta<\pi}F_\theta$ and $H_\theta = N_\theta\setminus\{(0,0)\}$, where $N_\theta$ is the line passing through the origin and is of the slope $\tan\theta$. Clearly $\overline{V}=\cup_{0\leq\theta<\pi}\overline{E}_\theta$ and $\{G\}\cup\{H_\theta : 0\leq\theta<\pi\}$ is a cover of $\overline{V}$ by open sets in $X$. We claim that $\overline{V}$ is not Lindel\"of. If possible let there exists a countable subfamily $\{G\}\cup\{H_{\theta_n} : n\in\mathbb{N}\}$ that covers $\overline{V}$. Now observe that if for each $n$ $\theta\neq\theta_n$, then $x_\theta\notin H_{\theta_n}$. By the construction of $G$ it is clear that $x_\theta\notin G$ and hence $x_\theta\notin\overline{V}$, a contradiction. Thus $\overline{V}$ can not be Lindel\"{o}f, contradicting our initial assumption that $X$ is locally Lindel\"{o}f.
\end{Ex}

Thus there is a regular space which does not satisfy any of the conditions of Theorem~\ref{T7}.

%\begin{Prop}
%\label{P15}
%Every locally Menger Lindel\"{o}f space $X$ is Menger.
%\end{Prop}
%\begin{proof}
%  Let $x\in X$. Choose an open set $U_x$ and a Menger subspace $Y_x$ such that $x\in U_x\subseteq Y_x$. Clearly the open cover $\{U_x : x\in X\}$ has a countable subfamily $\{U_{x_j} : j\in\mathbb{N}\}$  that covers $X$ i.e. $X=\cup_{j\in\mathbb{N}}U_{x_j}=
%  \cup_{j\in\mathbb{N}}Y_{x_j}$ and hence $X$ is Menger.
%\end{proof}

If $\mathcal P$ is a property, then $\non(\mathcal P)$ is the minimum cardinality of a set of reals that fails have to property $\mathcal P$. It is well known that $\non(\text{Menger})=\mathfrak{d}$ (see \cite{coc2}). Since in a Lindel\"{o}f space every locally Menger space is Menger, we obtain $\non(\text{locally Menger})=\mathfrak{d}$. Thus we can say that every locally Menger subspace of the Baire space $\mathbb{N}^\mathbb{N}$ is non dominating.

%\begin{Cor}
%\label{MT4}
%$\non(\text{locally Menger})=\mathfrak{d}$.
%\end{Cor}
%\begin{proof}
%Let $X$ be a set of reals with cardinality less than $\mathfrak{d}$. Clearly $X$ is Menger and hence locally Menger. We can obtain a set of reals $Y$ with cardinality $\mathfrak{d}$ such that $Y$ is not Menger. By Proposition~\ref{P15}, it follows that such $Y$ is not locally Menger. Hence $\non(\text{locally Menger})=\mathfrak{d}$.
%\end{proof}

%\begin{Cor}
%\label{PP1}
%Any locally Menger subspace of the Baire space $\mathbb{N}^\mathbb{N}$ is non dominating.
%\end{Cor}
%\begin{proof}
%Since any subspace of $\mathbb{N}^\mathbb{N}$ is Lindel\"{o}f, it follows that any locally Menger subspace of $\mathbb{N}^\mathbb{N}$ is Menger. Consequently any locally Menger subspace of $\mathbb{N}^\mathbb{N}$ is non dominating.
%\end{proof}
%
%\begin{Cor}
%\label{CC1}
%Any open continuous image of a locally Menger space into $\mathbb{N}^\mathbb{N}$ is non dominating.
%\end{Cor}

\subsection{Some observations on $P$-spaces}
Next we turn our attention to $P$-spaces. Recall that a $P$-space is a topological space in which countable intersection of open sets is open. An equivalent condition is that countable union of closed sets is closed. It is obvious that Lindel\"{o}f subspace of a Hausdorff $P$-space is closed. Also observe that every locally Menger subspace of a Hausdorff $P$-space $X$ is of the form $U\cap F$, where $U$ is open and $F$ is closed in $X$.

%We obtain the following characterization of locally Lindel\"{o}f subspaces in a Hausdorff $P$-space.

%\begin{Prop}[{Folklore}]
%\label{T14}
%Every locally Lindel\"{o}f subspace of a Hausdorff $P$-space $X$ is of the form $U\cap F$ where $U$ is open and $F$ is closed in $X$.
%\end{Prop}
%\begin{proof}
%Let $Y$ be a locally Lindel\"{o}f subspace of $X$. Let $y\in Y$. Choose an open set $V_y$ and a Lindel\"{o}f subspace $Z_y$ of $Y$ such that $y\in V_y\subseteq Z_y$. We can write $V_y=U_y\cap Y$ for some open set $U_y$ in $X$. Clearly $Z_y$ is closed in $X$ and $U_y\cap\overline{Y}\subseteq\overline{U_y\cap Y}$. Consequently $U_y\cap Y\subseteq U_y\cap\overline{Y}\subseteq Z_y$ and hence $Y=(\cup_{y\in Y}U_y)\cap\overline{Y}$.
%\end{proof}

%\begin{Cor}
%\label{C2}
%Every locally Menger subspace of a Hausdorff $P$-space $X$ is of the form $U\cap F$ where $U$ is open and $F$ is closed in $X$.
%\end{Cor}

The following equivalent formulations of the locally Menger property in Hausdorff $P$-spaces can be observed (compare with Theorem~\ref{T7}).
\begin{Th}
\label{T71}
Let $X$ be a Hausdorff $P$-space. The following assertions are equivalent.
\begin{enumerate}[wide,label={\upshape(\arabic*)},leftmargin=*]
\item $X$ is locally Menger.
\item For each $x\in X$ and for each open set $V$ with $x\in V$ there exist an open set $U$ and a Menger subspace $Y$ of $X$ such that $x\in U\subseteq Y\subseteq V$.
\item For each $x\in X$ there exists an open set $U$ with $x\in U$ such that $\overline{U}$ is Menger.
\item $X$ has a basis consisting of closed Menger neighbourhoods.
\end{enumerate}
\end{Th}
\begin{proof}
We only give proof of $(1)\Rightarrow(2)$. Let $x\in X$ and $V$ be an open set containing $x$. Choose an open set $U$ and a Menger subspace $Y$ of $X$ such that $x\in U\subseteq Y$. Then $Y$ is closed and hence $Y\cap(X\setminus V)$ is Menger. Now let $y\in Y\cap(X\setminus V)$. Then there exist disjoint open sets $U_y$ and $V_y$ containing $x$ and $y$ respectively. Now the cover $\{V_y : y\in Y\cap(X\setminus V)\}$ of $Y\cap(X\setminus V)$ by open sets in $X$ has a countable subfamily $\{V_y : y\in C\}$ that covers $Y\cap(X\setminus V)$, where $C$ is a countable subset of $Y\cap(X\setminus V)$. Choose $P=\cap_{y\in C}U_y$ and $Z=\cup_{y\in C}V_y$. Now it easy to observe that $P$ and $Z$ are disjoint open sets containing $x$ and $Y\cap(X\setminus V)$ respectively. Clearly $ W =P\cap \Int(Y)$ is an open set containing $x$ and $\overline{ W }\subseteq Y$. Thus $\overline{ W }$ is the required Menger subspace satisfying $\overline{ W }\subseteq V$.
%\noindent$(2)\Rightarrow(3).$ Let $x\in X$. Choose an open set $U$ and a Menger subspace $Y$ of $X$ such that $x\in U\subseteq Y$. Then $Y$ is closed and hence $\overline{U}\subseteq Y$. Clearly $\overline{U}$ is Menger and the conclusion follows.
%
%\noindent$(2)\Rightarrow(4).$ Let $x\in X$ and $V$ be an open set containing $x$. Choose an open set $U$ and a Menger subspace $Y$ of $X$ such that $x\in U\subseteq Y\subseteq V$. Then $Y$ is closed and consequently $\overline{U}$ is Menger. Clearly $X$ has a basis consisting of closed Menger neighbourhoods.
\end{proof}

The following result on embedding into Menger Hausdorff $P$-space can be obtained.
\begin{Th}
\label{T20}
Every locally Menger Hausdorff $P$-space can be densely embedded in a Menger Hausdorff $P$-space.
\end{Th}
\begin{proof}
 Let $(X,\tau)$ be a locally Menger Hausdorff $P$-space. It is enough to consider the case when $X$ is not Menger. Choose $p\notin X$ and define $X^\prime=X\cup\{p\}$. Consider the topology $\tau^\prime=\tau\cup\{U\subseteq X^\prime : X^\prime\setminus U\;\mbox{is a Menger subset of}\;X\}$ on $X^\prime$.

Clearly $X^\prime$ is a $P$-space. Claim that $X^\prime$ is Hausdorff. Let $x,\;y\in X^\prime$ be such that $x\neq y$. Without loss of generality assume that $x\in X$ and $y\notin X$. Clearly $y=p$ and we can choose an open set $U$ and a Menger subspace $Y$ of $X$ such that $x\in U\subseteq Y$. Consequently $U$ and $X^\prime\setminus Y$ are two disjoint open sets in $X^\prime$ such that $x\in U$ and $y\in X^\prime\setminus Y$.

We now show that $X^\prime$ is Menger. Let $(\mathcal{U}_n)$ be a sequence of open covers of $X^\prime$. For each $n$ choose $U_n\in\mathcal{U}_n$ such that $p\in U_n$. Then $X^\prime\setminus\cap_{n\in\mathbb{N}}U_n$ is a Menger subspace of $X$ and $(\mathcal{U}_n)$ is a sequence of covers of $X^\prime\setminus\cap_{n\in\mathbb{N}}U_n$ by open sets in $X^\prime$. For each $n$ choose a finite set $\mathcal{V}_n^\prime\subseteq\mathcal{U}_n$ such that the sequence $(\mathcal{V}_n^\prime)$ witnesses the Menger property of $X^\prime\setminus \cap_{n\in\mathbb{N}}U_n$. If we define for each $n$ $\mathcal{V}_n=\mathcal{V}_n^\prime\cup\{U_n\}$, then the sequence $(\mathcal{V}_n)$ guarantees for $(\mathcal{U}_n)$ that $X^\prime$ is Menger. It is obvious that $X$ is dense in $X^\prime$ and the inclusion mapping $\iota:X\to X^\prime$ is an embedding of $X$ into $X^\prime$.
\end{proof}

We obtain the following result on the character of a point in a locally Menger space.
\begin{Th}
\label{MT5}
The character $\chi(x,X)$ of a point $x$ in a locally Menger Hausdorff $P$-space $X$ is the smallest cardinal number of the form $|\mathcal{B}|^{\aleph_0}$, where $\mathcal{B}$ is a family of open subsets of $X$ such that $\cap\mathcal{B}=\{x\}$.
\end{Th}
\begin{proof}
Let $\{x\}=\cap_{\alpha\in\Lambda}V_\alpha$, where each $V_\alpha$ is open in $X$ and $|\Lambda|\leq\kappa$. It is sufficient to show that $\chi(x,X)\leq\kappa^{\aleph_0}$. Since $X$ is locally Menger, there exist an open subset $V$ of $X$ and a Menger subspace $Y$ of $X$ such that $x\in V\subseteq Y$. By Theorem~\ref{T71}, we have for each $\alpha\in\Lambda$ an open subset $U_\alpha$ of $X$ such that $x\in U_\alpha\subseteq \overline{U}_\alpha\subseteq V_\alpha\cap V$ and $\overline{U}_\alpha$ is Menger. Clearly $\{x\}=\cap_{\alpha\in\Lambda}\overline{U}_\alpha$. Let $U$ be an open set in $X$ containing $x$. Now $\cap_{\alpha\in\Lambda}\overline{U}_\alpha\subseteq U$ and hence $\cap_{\alpha\in\Lambda}\overline{U}_\alpha\subseteq U\cap\overline{U}_{\alpha_0}$ for some fixed $\alpha_0\in\Lambda$. It follows that $\overline{U}_{\alpha_0}\setminus(U\cap\overline{U}_{\alpha_0})\subseteq
\cup_{\alpha\in\Lambda}(\overline{U}_{\alpha_0}\setminus\overline{U}_\alpha)$ i.e. $\{\overline{U}_{\alpha_0}\setminus\overline{U}_\alpha : \alpha\in\Lambda\}$ is a cover of $\overline{U}_{\alpha_0}\setminus(U\cap\overline{U}_{\alpha_0})$ by open sets in $\overline{U}_{\alpha_0}$. Clearly $\overline{U}_{\alpha_0}\setminus(U\cap\overline{U}_{\alpha_0})$ is Menger and hence there exists a countable subfamily $\{\overline{U}_{\alpha_0}\setminus\overline{U}_{\alpha_n} : n\in\mathbb{N}\}$ that covers $\overline{U}_{\alpha_0}\setminus(U\cap\overline{U}_{\alpha_0})$. Consequently $\cap_{n\in\mathbb{N}}\overline{U}_{\alpha_n}\subseteq U$. Thus the collection of all countable intersections of members of $\{U_\alpha : \alpha\in\Lambda\}$ forms a local base for $x$. Since the cardinality of such local base for $x$ does not exceed $\kappa^{\aleph_0}$, we have $\chi(x,X)\leq\kappa^{\aleph_0}$.
\end{proof}

\begin{Lemma}
Let $X$ be a Hausdorff $P$-space.
\begin{enumerate}[label={\upshape(\arabic*)},ref={\theLemma (\arabic*)}]
\item \label{ML2}
There exists a continuous bijective mapping of $X$ onto a Hausdorff $P$-space $Y$ such that $w(Y)\leq nw(X)^{\aleph_0}$.

\item \label{MC2}
If in addition $X$ is Lindel\"{o}f, then $w(X)\leq nw(X)^{\aleph_0}$.
\end{enumerate}
\end{Lemma}
\begin{proof}
\noindent$(1).$ Let $nw(X)=\kappa$ and $\mathcal{N}$ be a network for $X$ with $|\mathcal{N}|=\kappa$. Let $\tau$ be the topology on $X$. We search for pairs $A_1$, $A_2$ of members of $\mathcal{N}$ such that there exist two disjoint open sets $U_1$, $U_2$ $\in\tau$ containing $A_1$ and $A_2$ respectively, and for every such pair select some open sets $U_1$, $U_2$. Let $\mathcal{A}$ be the collection of all open sets obtained in the above process and $\mathcal{B}$ be the collection of all possible countable intersections of members of $\mathcal{A}$. Clearly $\mathcal{B}$ is a base for a topology on $X$ and let $\tau^*$ be the topology on $X$ generated by the base $\mathcal{B}$, we denote this space by $Y$.  By construction, $Y$ is a $P$-space. Since $X$ is a Hausdorff space, $Y$ is also so. Clearly $\tau$ is finer than $\tau^*$ as $X$ is a $P$-space. We now choose the identity mapping from $X$ onto $Y$. Since the cardinality of $\mathcal{B}$ does not exceed $\kappa^{\aleph_0}$, we can conclude that $w(Y)\leq nw(X)^{\aleph_0}$.\\

\noindent$(2).$  By (1), there is a continuous injective mapping $f$ from $X$ onto a Hausdorff P-space $Y$ with $w(Y)\leq nw(X)^{\aleph_0}$. We claim that such $f$ is closed. Let $A$ be a closed set in $X$. Since Mengerness is preserved under closed subspaces and continuous mappings, $f(A)$ is closed in $Y$. Thus $f$ is a homeomorphism and $w(X)\leq nw(X)^{\aleph_0}$.
\end{proof}

\begin{Th}
\label{MT6}
If $X$ is a locally Menger Hausdorff $P$-space, then $w(X)\leq nw(X)^{\aleph_0}$.
\end{Th}
\begin{proof}
  Let $nw(X)=\kappa$ and $\mathcal{N}$ be a network for $X$ with $|\mathcal{N}|=\kappa$. Since $X$ is locally Menger, for each $x\in X$ there exist an open subset $U_x$ of $X$ and a Menger subspace $Y_x$ of $X$ such that $x\in U_x\subseteq Y_x$. By Theorem~\ref{T71}, for each $x\in X$ we can choose an open set $V_x$ in $X$ such that $x\in V_x\subseteq\overline{V}_x\subseteq U_x$ and $\overline{V}_x$ is Menger. Again for each $x\in X$ we can find $A_x\in\mathcal{N}$ such that $x\in A_x\subseteq V_x$ and hence each $\overline{A}_x$ is Menger. Thus the family $\mathcal{N}=\{A_\alpha\}_{\alpha\in\Lambda}$ has a subcollection consisting of members whose closures are Menger forms a cover of $X$. Since $X$ is regular and locally Menger, for each $y\in\overline{A}_\alpha$ similarly we can find an open set $W_y$ in $X$ containing $y$ such that $\overline{W}_y$ is Menger. Now $\{W_y : y\in\overline{A}_\alpha\}$ is a cover of $\overline{A}_\alpha$ by open sets in $X$ and hence there is a countable subfamily $\{W_{y_n} : n\in\mathbb{N}\}$ such that $\overline{A}_\alpha\subseteq\cup_{n\in\mathbb{N}}W_{y_n}$. Choose $V_\alpha=\cup_{n\in\mathbb{N}}W_{y_n}$. Clearly $\overline{V}_\alpha=\cup_{n\in\mathbb{N}}\overline{W}_{y_n}$  and hence $\overline{V}_\alpha$ is Menger. It follows that for each $\alpha\in\Lambda$ there exists an open set $V_\alpha$ in $X$ such that $\overline{A}_\alpha\subseteq V_\alpha$ and $\overline{V}_\alpha$ is Menger. Since $nw(\overline{V}_\alpha)\leq nw(X)=\kappa$, by Lemma~\ref{MC2} we have $w(\overline{V}_\alpha)\leq\kappa^{\aleph_0}$ and hence $w(V_\alpha)\leq\kappa^{\aleph_0}$ i.e. there exists a base $\mathcal{B}_\alpha$ for the subspace $V_\alpha$ of $X$ with $|\mathcal{B}_\alpha|\leq\kappa^{\aleph_0}$. Now it is easy to observe that $\cup_{\alpha\in\Lambda}\mathcal{B}_{\alpha}$ is a base for  $X$ with cardinality less than or equal to $\kappa^{\aleph_0}$. Clearly $w(X)\leq\kappa^{\aleph_0}$ i.e. $w(X)\leq nw(X)^{\aleph_0}$.
\end{proof}

\begin{Cor}
\label{MC3}
If $X$ is a locally Menger Hausdorff $P$-space, then $w(X)\leq |X|^{\aleph_0}$.
\end{Cor}

\section{Preservation like properties}
\subsection{Preservation properties}
We now exhibit some preservation properties in locally Menger spaces. It is to be noted that the locally Menger property is not hereditary, not even in a Tychonoff space. The space $\mathbb{P}$ of irrationals is not locally Menger while the real line $\mathbb{R}$ is a Tychonoff locally compact space. Observe that the locally Menger property is preserved under $F_\sigma$-subspaces and also in a regular space, this property is preserved under locally closed subspaces.

%\begin{Prop}
%\label{T8}
%\hfill
%\begin{enumerate}[wide,label={\upshape(\arabic*)},leftmargin=*]
%\item Locally Menger property is preserved under $F_\sigma$-subspaces.
%\item In a regular space, locally Menger property is preserved under locally closed subspaces.
%\end{enumerate}
%\end{Prop}
%\begin{proof}
%%$(1).$ Let $X$ be a locally Menger space and $Y$ be an $F_\sigma$-subspace of $X$. Let $y\in Y$. Choose an open set $U$ and a Menger subspace $Z$ of $X$ such that $y\in U\subseteq Z$. Thus $y\in U\cap Y\subseteq Z\cap Y$. Since Menger property is preserved under $F_\sigma$ subspaces, $Z\cap Y$ is Menger and hence $Y$ is locally Menger.
%
%We only present proof of $(2)$. Let $Y$ be a locally closed subset of a regular locally Menger space $X$. We can write $Y=U\cap C$ where $U$ is open in $X$ and $C$ is closed in $X$. Let $y\in Y$ and $V$ be an open set in $Y$ containing $y$, say $V=O\cap Y$ for some open set $O$ in $X$. Choose an open set $W$ and a Menger subspace $Z$ of $X$ such that $y\in W\subseteq Z\subseteq O\cap U$. Now $y\in W\cap Y\subseteq Z\cap Y\subseteq V$ and clearly $Z\cap Y$ is Menger. Hence $Y$ is locally Menger.
%\end{proof}

We now provide an example of a non regular space in which the locally Menger property is not preserved under locally closed subspaces.
%\begin{Ex}
%\label{EE2}
Let $X=\mathbb{N}^\mathbb{N}$ be the Baire space and $aX$ be the one point compactification of $X$. Clearly  $aX$ is locally Menger but not regular. It is also clear that $X$ is a locally closed subspace of $aX$ which is not locally Menger.
%\end{Ex}
It is also worth mentioning that the Menger property is not preserved under locally closed subspaces, not even in regular spaces. The space $[0,\omega_1]$ is regular and also Menger, but its locally closed subspace $[0,\omega_1)$ is not Menger (as it is not Lindel\"{o}f).

Note that the Menger property is an invariant under continuous mappings \cite[Theorem 3.1]{coc2}, but the following example shows that the locally Menger property is not an invariant under continuous mappings.
%\begin{Ex}
Let $X=\mathbb{R}^2$ with the discrete topology and $Y=\mathbb{R}^2$ with the radial interval topology. Then $X$ is locally Menger but $Y$ is not so (see Example~\ref{P10}). By considering the identity mapping $i:X\to Y$, we can conclude that the locally Menger property is not an invariant under continuous mappings.
%\end{Ex}

We now observe preservation of the locally Menger property under certain mappings for the next couple of results.
We first recall the following definitions (from \cite{Bi-quotient}).
A surjective continuous mapping $f:X\to Y$ is said to be weakly perfect if $f$ is closed and $f^{-1}(y)$ is Lindel\"{o}f for every $y\in Y$.
%A surjective continuous mapping $f:X\to Y$ is said to be bi-quotient
Also a surjective continuous mapping $f:X\to Y$ is said to be bi-quotient if whenever $y\in Y$ and $\mathcal{U}$ is a cover of $f^{-1}(y)$ by open sets in $X$, then finitely many $f(U)$ with $U\in\mathcal{U}$ cover some open set containing $y$ in $Y$. Clearly surjective continuous open (and also perfect) mappings are bi-quotient.

\begin{Th}
\label{T15}
The locally Menger property is an invariant under weakly perfect as well as bi-quotient mappings.
%Let $X$ be locally Menger.
%\begin{enumerate}[wide,label={\upshape(\arabic*)},
%ref={\theTh (\arabic*)},leftmargin=*]
%\item \label{T1502}  If $f:X\to Y$ is  weakly perfect, then $Y$ is locally Menger.
%\item \label{TD1}If $f:X\to Y$ is  bi-quotient, then $Y$ is locally Menger.
%\end{enumerate}
\end{Th}
\begin{proof}

Let $f:X\to Y$ be weakly perfect, where $X$ is locally Menger. Let $y\in Y$. For each $x\in f^{-1}(y)$, choose an open set $U_x$ and a Menger subspace $Z_x$ of $X$ such that $x\in U_x\subseteq Z_x$. Since $f^{-1}(y)$ is Lindel\"of, there is a countable subcollection $\{U_{x_n} : n\in\mathbb{N}\}$ of $\{U_x : x\in f^{-1}(y)\}$ that covers $f^{-1}(y)$. Consequently $f^{-1}(y)\subseteq
\cup_{n\in\mathbb{N}}Z_{x_n}$. Moreover $y\in Y\setminus f(X\setminus\cup_{n\in\mathbb{N}}U_{x_n})\subseteq f(\cup_{n\in\mathbb{N}}Z_{x_n})$. Since $f$ is closed, $Y\setminus f(X\setminus\cup_{n\in\mathbb{N}}U_{x_n})$ is an open set in $Y$ containing $y$. Also since $\cup_{n\in\mathbb{N}}Z_{x_n}$ is Menger, $f(\cup_{n\in\mathbb{N}}Z_{x_n})$ is a Menger subspace of $Y$.\\

Next suppose that $f:X\to Y$ is a bi-quotient mapping, where $X$ is locally Menger. For each $x\in X$, choose an open set $U_x$ and a Menger subspace $Z_x$ of $X$ such that $x\in U_x\subseteq Z_x$. Let $y\in Y$ and consider $f^{-1}(y)$.
%Now $\{U_x : x\in X\}$ is a cover of $f^{-1}(y)$ by open sets in $X$ and
Since $f$ is a bi-quotient mapping, there exist a finite subset $\{U_{x_i} : 1\leq i\leq k\}$ of $\{U_x : x\in X\}$ and an open set $V$ containing $y$ in $Y$ such that $V\subseteq\cup_{i=1}^k f(U_{x_i})$. Clearly $y\in\Int f(\cup_{i=1}^k U_{x_i})$ and $f(\cup_{i=1}^k Z_{x_i})$ is a Menger subspace of $Y$ with $\Int f(\cup_{i=1}^k U_{x_i})\subseteq f(\cup_{i=1}^k Z_{x_i})$. Thus $Y$ is locally Menger.
\end{proof}

\begin{Cor}
\label{T1501}
If $f:X\rightarrow Y$ is a perfect (or an open continuous) mapping from a locally Menger space $X$ onto a space $Y$, then $Y$ is locally Menger.
\end{Cor}

\begin{Prop}
\label{T21-1}
The locally Menger property is an inverse invariant under injective closed continuous mappings.
%Locally Menger property is an inverse invariant under injective closed continuous mappings (as well as under injective Menger-covering mappings).
\end{Prop}
%\begin{proof}
%We give only proof for the case of injective Menger-covering mappings.
%Let $f:X\rightarrow Y$ be an injective Menger-covering mapping from a space $X$ onto a locally Menger space $Y$. Let $x\in X$. We can find an open subset $U$ and a Menger subspace $Z$ of $Y$ such that $f(x)\in U\subseteq Z$. Since $f$ is a Menger-covering mapping, there exists a Menger subspace $Z^\prime$ of $X$ such that $f(Z^\prime)=Z$ i.e. $Z^\prime=f^{-1}(Z)$. Clearly $x\in f^{-1}(U)\subseteq Z^\prime$ and hence $X$ is locally Menger.
%\end{proof}
Note that the Menger (respectively, locally Menger) property is not an inverse invariant under open as well as continuous mappings.
 Take $X=X_1\times X_2$, where $X_1=[0,\omega_1)$, $X_2=[0,\omega_1]$. Consider the projection $p_2:X\to X_2$. Clearly $p_2$ is a continuous open mapping and $X_2$ is Menger whereas $X$ is not Menger as $X_1$ is not.
Also consider $X=X_1\times X_2$, where $X_1=\Psi(\mathcal{A})$ the Isbell-Mr\'{o}wka space and $X_2=\mathbb{R}^2$ with radial interval topology. Let $p_1:X\to X_1$ be the projection onto first coordinate. Clearly $p_1$ is both continuous and open and $X_1$ is locally Menger (see Example~\ref{E1}). Now $X$ is not locally Menger as $X_2$ is not (see Example~\ref{P10}).

Let $\{X_\alpha : \alpha\in\Lambda\}$ be a family of topological spaces. Let $X=\oplus_{\alpha\in\Lambda}X_\alpha$ be the topological sum which is defined as $\oplus_{\alpha\in\Lambda}X_\alpha=\cup_{\alpha\in\Lambda}\{(x,\alpha) : x\in X_\alpha\}$. For each $\alpha\in\Lambda$, let $\varphi_\alpha:X_\alpha\to X$ be defined by $\varphi_\alpha(x)=(x,\alpha)$. The topology on $X$ is defined as follows. A subset $U$ of $X$ is open in $X$ if and only if $\varphi_\alpha^{-1}(U)$ is open in $X_\alpha$ for each $\alpha\in\Lambda$. Also if for each $\alpha\in\Lambda$ the space $X_\alpha$ is homeomorphic to a fixed space $Y$, then the topological sum $\oplus_{\alpha\in\Lambda}X_\alpha$ is homeomorphic to $Y\times\Lambda$, where $\Lambda$ has the discrete topology.

Clearly if $X=\cup_{\alpha\in\Lambda}X_\alpha$, where each $X_\alpha$ is an open locally Menger subspace of $X$, then $X$ is locally Menger and also the topological sum $\oplus_{\alpha\in\Lambda}X_\alpha$ is locally Menger if and only if each $X_\alpha$ is locally Menger. We cannot replace `locally Menger' by `Menger' in these assertions. Take $X=[0,\omega_1)$. Then $\{[0,\alpha) : \alpha<\omega_1\}$ is an open cover of $X$. For each $\alpha<\omega_1$, $[0,\alpha)$ is Menger but $X$ is not. Also if
 $Y$ is the topological sum of $\omega_1$ copies of $[0,1]$, then $Y$ is homeomorphic to $[0,1]\times D$, where $D$ is a discrete space with $|D|=\omega_1$. Verify that $[0,1]\times D$ is not Menger. Thus $Y$ is the topological sum of Menger spaces which is itself not Menger.

Since there exist two Menger sets of reals $X$ and $Y$ such that the product $X\times Y$ is not Menger (see \cite[Theorem 2.10]{PMS}), it follows that product of two locally Menger spaces need not be locally Menger.

Recall that the product of a Menger space and a $\sigma$-compact space is again Menger \cite[Proposition 4.7]{CST}. If we replace $\sigma$-compact  by locally $\sigma$-compact, then the product space need not be Menger. For example, $X_1=[0,\omega_1)$ is locally $\sigma$-compact and $X_2=[0,\omega_1]$ is Menger but $X_1\times X_2$ is not Menger. However if $X$ is locally Menger and $Y$ is locally $\sigma$-compact, then $X\times Y$ is locally Menger.
%\begin{Prop}
%\label{T28}
%If $X$ is locally Menger and $Y$ is locally $\sigma$-compact, then $X\times Y$ is locally Menger.
%\end{Prop}
%\begin{proof}
%Let $X$ be locally Menger and $Y$ be locally $\sigma$-compact. Let $(x,y)\in X\times Y$. We can find an open set $U$ in $X$ and a Menger subspace $Y$ of $X$ such that $x\in U\subseteq Y$. Also we can choose an open set $V$ and a $\sigma$-compact set $K$ in $Y$ such that $y\in V\subseteq K$. Clearly $Y\times K$ is Menger and $(x,y)\in U\times V\subseteq Y\times K$. Hence $X\times Y$ is locally Menger.
%\end{proof}
%Using the projection mappings the following result is immediate.
If the Cartesian product $\prod_{\alpha\in\Lambda} X_\alpha$ is locally Menger, then each $X_\alpha$ is locally Menger. We also obtain the following.
\begin{Prop}
\label{PP3}
If the Cartesian product $\prod_{\alpha\in\Lambda} X_\alpha$ is locally Menger, then $X_\alpha$ is Menger for all but finitely many $\alpha$. %\begin{enumerate}[wide,label={\upshape(\arabic*)},
%leftmargin=*]
%  \item  each $X_\alpha$ is locally Menger and
%  \item  $X_\alpha$ is Menger for all but finitely many $\alpha$.
%  %there exists a finite set $\Lambda_0\subseteq\Lambda$ such that $X_\alpha$ is Menger for each $\alpha\in\Lambda\setminus\Lambda_0$.
%\end{enumerate}
\end{Prop}
\begin{proof}
%We give only proof of $(2)$.
Let $x\in\prod_{\alpha\in\Lambda} X_\alpha$. Choose an open set $U$ and a Menger subspace $Y$ of $\prod_{\alpha\in\Lambda} X_\alpha$ such that $x\in U\subseteq Y$. We can find a basic member $\prod_{\alpha\in\Lambda} V_\alpha$ of $\prod_{\alpha\in\Lambda} X_\alpha$ such that for some finite set $\Lambda_0\subseteq\Lambda$,  $V_\alpha=X_\alpha$ for each $\alpha\in\Lambda\setminus\Lambda_0$ and $\prod_{\alpha\in\Lambda} V_\alpha\subseteq U$. If for each $\alpha\in\Lambda$, $p_\alpha:\prod_{\alpha\in\Lambda} X_\alpha\to X_\alpha$ is the projection mapping, then for each $\alpha\in\Lambda\setminus\Lambda_0$, $P_\alpha(Y)=X_\alpha$ is Menger.
\end{proof}

\begin{Th}
\label{TP1} If $X=\cup_{\alpha\in\Lambda}X_\alpha$, where each $X_\alpha$ is a closed locally Menger subspace of $X$ and the collection $\{X_\alpha : \alpha\in\Lambda\}$ is locally finite in $X$, then $X$ is locally Menger.
\end{Th}
\begin{proof}
Let $Y=\oplus_{\alpha\in\Lambda}X_\alpha$. Then $Y$ is locally Menger. Define $f:Y\to X$ by $f(x,\alpha)=x$ and also for each $\alpha$ define $\varphi_\alpha:X_\alpha\to Y$  by $\varphi_\alpha(x)=(x,\alpha)$.
%We now show that $f$ is a perfect mapping.
%%If $U$ is open in $X$ then clearly, $f^{-1}(U)=\cup_{\alpha\in\Lambda}\varphi_\alpha(U\cap X_\alpha)$. Since for each $\alpha$, $\varphi_\alpha(U\cap X_\alpha)$ is open in $Y$ we have $f^{-1}(U)$ is open in $Y$ and hence $f$ is continuous.
%Clearly, $f$ is continuous.
%Let $F$ be  closed in $Y$.  Since for each $\alpha$ $\varphi_\alpha^{-1}(F)$ is closed in $X$, it follows that $f(F)=\cup_{\alpha\in\Lambda}\varphi_\alpha^{-1}(F)$ is also closed in $X$ as $\{X_\alpha : \alpha\in\Lambda\}$ is locally finite in $X$. Thus $f$ is closed. Now let $x\in X$. Choose an open set $V$ containing $x$ in $X$ that intersects only finitely many members of $\{X_\alpha : \alpha\in\Lambda\}$, say $\{X_{\alpha_i} : 1\leq i\leq k\}$.
%%This implies that $x\notin X_\alpha$ for all $\alpha\in\Lambda\setminus\{\alpha_i : 1\leq i\leq k\}$.
%Now $f^{-1}(x)=\oplus_{\{\alpha_i : 1\leq i\leq k\}}\{x\}$ and hence $f^{-1}(x)$ is a compact subspace of $Y$. Thus $f$ is a perfect mapping and the conclusion now follows from Corollary~\ref{T1501}.
Now for each closed set $F$ in $Y$, $f(F)=\cup_{\alpha\in\Lambda}\varphi_\alpha^{-1}(F)$ is also closed in $X$. Again if any $x\in X$ is chosen, then by local finiteness of the given collection, there is an open set $V$ containing $x$ which intersects only with a finite subcollection, say $\{X_{\alpha_i} : 1\leq i\leq k\}$. Also we can write $f^{-1}(x)=\oplus_{\{\alpha_i : 1\leq i\leq k\}}\{x\}$. Clearly $f$ is a perfect mapping and the conclusion now follows from Corollary~\ref{T1501}.
\end{proof}

%\begin{Cor}
%\label{CP1}
%If $X=\cup_{i=1}^k X_i$, where each $X_i$ is a closed locally star-Menger subspace of $X$, then $X$ is locally star-Menger.
%\end{Cor}
%Just as Theorem~\ref{TP1}, with the help of the above lemma we can similarly verify the following result.
A similar assertion also holds in the context of $P$-spaces. Since the union of the members in a locally countable collection of closed sets in a $P$-space is again a closed set, we have the following.
%First we need the following observation.
%\begin{Lemma}[Folklore]
%Let $X$ be a $P$-space. If $\{X_\alpha : \alpha\in\Lambda\}$ is a locally countable family of closed sets in $X$, then $\cup_{\alpha\in\Lambda} X_\alpha$ is also closed in $X$.
%\end{Lemma}
%\begin{proof}
%Let $Y=\cup_{\alpha\in\Lambda} X_\alpha$ and $x\in X\setminus Y$. Since $\{X_\alpha : \alpha\in\Lambda\}$ is locally countable in $X$, we can find an open subset $U$ of $X$ containing $x$ which intersects only countably many members of $\{X_\alpha : \alpha\in\Lambda\}$, say $\{X_{\alpha_n} : n\in\mathbb{N}\}$. It follows that $U\setminus Y=U\setminus\cup_{n\in\mathbb{N}}X_{\alpha_n}$ and hence $U\setminus Y$ is an open subset of $X$ containing $x$ which does not intersect $Y$. Thus $Y$ is closed in $X$.
%\end{proof}
\begin{Th}
\label{TP2}
If $X=\cup_{\alpha\in\Lambda}X_\alpha$ is a $P$-space, where each $X_\alpha$ is a closed locally Menger subspace of $X$ and the collection $\{X_\alpha : \alpha\in\Lambda\}$ is locally countable in $X$, then $X$ is locally Menger.
\end{Th}

\subsection{MG-spaces and some related observations}
We start with a basic characterization of open (respectively, closed) sets in locally Menger spaces.
\begin{Prop}
\label{T25}
    Let $X$ be locally Menger. A subset $Y$ is open (respectively, closed) in $X$ if and only if for each Menger subspace $M$ of $X$, $Y\cap M$ is open (respectively, closed) in $M$.
 \end{Prop}

  %\begin{proof}
%  We give only proof of the reverse implication.
%  Let $x\in Y$. Choose an open set $V$ and a Menger subspace $M$ such that $x\in V\subseteq M$. By the given condition $Y\cap M$ is open in $M$ and so $Y\cap V$ is open in $X$. Therefore, $Y$ is open in $X$.\\
%  A Similar argument holds for closed set characterization.
%  \end{proof}

The above result instigates the following definition.
\begin{Def}
A space $X$ is said to be Menger generated (in short, MG-space) if it satisfies the following condition.\\
   A subset $U$ is open in $X$ provided that  $U\cap M$ is open in $M$ for every Menger subspace $M$ of $X$.
\end{Def}
An equivalent condition is that for each $F\subseteq X$, the set $F$ is closed in $X$ provided that $F\cap M$ is closed in $M$ for any Menger subspace $M$ of $X$.
By Proposition~\ref{T25}, every locally Menger space is a MG-space.

Recall that a space $X$ is said to be a sequential if a subset of $X$ is closed if and only if together with any sequence it contains all its limits.

\begin{Ex}
\label{PP4}
A sequential space $X$ is a MG-space. The reason is as follows.
Let $F$ be subset of $X$ such that $F\cap M$ is closed in $M$ for each Menger subspace $M$ of $X$. If $F$ is not closed in $X$, then we can find a sequence $(x_n)$ of points of $F$ such that there exists a limit $x_0$ of $(x_n)$ with $x_0\notin F$. Clearly $Y=\{x_0,x_1,x_2,\ldots\}$ is a Menger subspace of $X$ but $F\cap Y$ is not closed in $Y$, a contradiction. Therefore $X$ is a MG-space.
\end{Ex}

We now make certain investigations on MG-spaces. The next two results will be used subsequently.
\begin{Lemma}[{Folklore}]
\label{LL1}
If $X$ is  Lindel\"{o}f, then for any $P$-space $Y$ the projection $p:X\times Y\to Y$ is closed.
\end{Lemma}

\begin{Lemma}[{cf.\cite{Engelking}}]
\label{LL2}
Let $f:X\to Y$ be a mapping.
\begin{enumerate}[wide,label={\upshape(\arabic*)},
ref={\theLemma (\arabic*)},
leftmargin=*]
  \item\label{LL201} If $f$ is  closed (respectively, open), then for any subspace $Z$ of $Y$ the restriction $f_Z:f^{-1}(Z)\to Z$ is also closed (respectively, open).
  \item\label{LL202} If $f$ is a quotient mapping, then for any subset $B$ of $Y$ which is either closed or open, the restriction $f_B:f^{-1}(B)\to B$ is also a quotient mapping.
  \end{enumerate}
\end{Lemma}

First we present another delineation of a MG-space.
\begin{Th}
\label{TT1}
A space is a MG-space if and only if it is a quotient space of some locally Menger space.
\end{Th}
\begin{proof}
Let $\{M_\alpha :\alpha\in\Lambda\}$ be the collection of all Menger subspaces of a MG-space $X$. Choose $Y=\oplus_{\alpha\in\Lambda}M_\alpha$. Then $Y$ is locally Menger and $f:Y\to X$ defined by $f(x,\alpha)=x$ is a quotient map.

Conversely suppose that $q:X\to Y$ is a quotient map, where $X$ is locally Menger. Let $U$ be a subset of $Y$ such that $U\cap M$ is open in $M$ for each Menger subspace $M$ of $Y$. Let $x\in q^{-1}(U)$. Choose an open set $V$ and a Menger subspace $M$ of $X$ so that $x\in V\subseteq M$. Now $q(M)$ is a Menger subspace of $Y$ and hence $U\cap q(M)$ is open in $q(M)$. Clearly there is an open set $W$ in $Y$ such that $U\cap q(M)=W\cap q(M)$. This shows that $x\in q^{-1}(W)\cap M$ and hence $x\in q^{-1}(W)\cap V$. Also $q^{-1}(W)\cap V$ is an open set in $X$ containing $x$ such that $q^{-1}(W)\cap V\subseteq q^{-1}(U)$. Evidently $q^{-1}(U)$ is open in $X$. Thus $U$ is open in $Y$ and $Y$ is a MG-space.
\end{proof}

\begin{Cor}
\label{PP8}
If $f:X\to Y$ is a quotient mapping from a MG-space $X$ onto a space $Y$, then $Y$ is a MG-space.
\end{Cor}
%\begin{proof}
%Since $X$ is a MG-space, we have a locally Menger space $Z$ and a quotient mapping $q:Z\to X$. The composition $fq:Z\to Y$ is also a quotient mapping and hence $Y$ is a MG-space.
%\end{proof}
By Lemma~\ref{LL2}, we have the following observation.
\begin{Cor}
\label{PP9}
Every closed subspace of a MG-space is a MG-space.
\end{Cor}
%\begin{proof}
% Let $Y$ be a closed subspace of a MG-space $X$. We can find a locally Menger space $Z$ and a quotient mapping $q:Z\to X$. By Lemma~\ref{LL2}, $q_Y:q^{-1}(Y)\to Y$ is a quotient mapping. Clearly $q^{-1}(Y)$ is  locally Menger  and hence $Y$ is a MG-space.
%\end{proof}

The following example suggests that a MG-space need not be locally Menger.
\begin{Ex}
\label{EE01}
Let $X$ be the topological sum of $\omega_1$ copies of $[0,1]$. Then $X$ is locally Menger. Let $Y$ be the quotient space of $X$ obtained by identifying all the zero points to a single point $y_0$. The mapping $q:X\to Y$ given by
\begin{equation*}
q(x,\alpha)=
\begin{cases}
y_0 & \text{if~} x=0\\
(x,\alpha) & \text{otherwise}
\end{cases}
\end{equation*}
is a quotient mapping. Clearly $Y$ is a MG-space by Theorem~\ref{TT1}. We now show that $Y$ is not locally Menger. It is sufficient to show that $Y$ is not locally Lindel\"{o}f. Assume the contrary. Then we can find an open subset $U$ and a Lindel\"{o}f subspace $L$ of $Y$ such that $y_0\in U\subseteq L$. Clearly for each $\alpha\in\omega_1$ there exists $a_\alpha\in(0,1]$ such that $\{y_0\}\cup\oplus_{\alpha\in\omega_1}(0,a_\alpha)\subseteq U$. Choose $V=\{y_0\}\cup\oplus_{\alpha\in\omega_1}(0,b_\alpha)$, where $b_\alpha\in(0,a_\alpha)$ for all $\alpha\in\omega_1$. Clearly  $\overline{V}=\{y_0\}\cup\oplus_{\alpha\in\omega_1}(0,b_\alpha]\subseteq\{y_0\}\cup\oplus_{\alpha\in\omega_1}(0,a_\alpha)$ and hence $\overline{V}$ is Lindel\"{o}f. For each $\alpha\in\omega_1$ choose $U_\alpha=\{(x,\alpha) : x\in(0,a_\alpha)\}$. Also take $W=\{y_0\}\cup\oplus_{\alpha\in\omega_1}(0,c_\alpha)$, where $c_\alpha\in(0,b_\alpha)$ for all $\alpha\in\omega_1$. Then $\mathcal{U}=\{U_\alpha : \alpha\in\omega_1\}\cup\{W\}$ is a cover of $\overline{V}$ by open sets in $Y$. It follows that there is no countable subfamily of $\mathcal{U}$ that covers $\overline{V}$. Thus $\overline{V}$ is not Lindel\"{o}f, contradicting our assumption that $Y$ is locally Lindel\"{o}f.
\end{Ex}
It also follows from the above example that the locally Menger property is not an invariant under quotient mappings.

\begin{Th}
\label{TT2}
Let $X$ be a locally Menger Hausdorff $P$-space and $Y$ be a $P$-space. If $q:Y\to Z$ is a quotient mapping, then $f:X\times Y\to X\times Z$ given by $f(x,y)=(x,q(y))$ is also a quotient mapping.
\end{Th}

\begin{proof}
Let $W\subseteq X\times Z$ be such that $f^{-1}(W)$ is  open in $X\times Y$. Let $(x,z)\in W$. Choose $y\in Y$ such that $y\in q^{-1}(z)$. Clearly $(x,y)\in f^{-1}(W)$ and hence there exist open sets $\tilde{U}$ and $\tilde{V}$ in $X$ and $Y$ respectively such that $(x,y)\in\tilde{U}\times\tilde{V}\subseteq f^{-1}(W)$. Choose an open set $U$ containing $x$ in $X$  such that $\overline{U}$ is a Menger subspace of $X$ with $\overline{U}\subseteq\tilde{U}$ and hence $\overline{U}\times\{y\}\subseteq f^{-1}(W)$. It is easy to see that for any $v\in Y$, if $\overline{U}\times\{v\}\subseteq f^{-1}(W)$, then $\overline{U}\times q^{-1}q(v)\subseteq f^{-1}(W)$. Clearly $\overline{U}\times q^{-1}(z)\subseteq f^{-1}(W)$. Choose $V=\{w\in Z : \overline{U}\times q^{-1}(w)\subseteq f^{-1}(W)\}$. We then have $(x,z)\in U\times V\subseteq W$. The proof will be furnished if we can show that $V$ is an open subset of $Y$. Now $q^{-1}(V)=\{v\in Y : \overline{U}\times\{v\}\subseteq f^{-1}(W)\}$. By Lemma~\ref{LL1}, the projection mapping $p:\overline{U}\times Y\to Y$ is closed. Clearly $q^{-1}(V)$ is equal to the complement of $p((\overline{U}\times Y)\setminus f^{-1}(W))$ and hence $V$ is open in $Y$.
\end{proof}

Next two results can be easily verified.
\begin{Prop}
\label{PP5}
Let $f:X\to Y$ be a mapping.
\begin{enumerate}[wide,label={\upshape(\arabic*)},
leftmargin=*]
\item If $X$ is a MG-space, then $f$ is continuous if and only if for every Menger subspace $M$ of $X$ the restriction $f|_M: M\to Y$ is continuous.
\item If $Y$ is a MG-space, then $f$ is closed (respectively, open) if and only if for any Menger subspace $M$ of $Y$ the restriction $f_M:f^{-1}(M)\to M$ is closed (respectively, open).
\end{enumerate}
\end{Prop}
%\begin{proof}
%\noindent$(1).$ We give only proof for the converse part. Let $V$ be open in $Y$. Clearly $(f|_M)^{-1}(V)$ is  open in $M$, for each Menger subspace $M$ of $X$. Clearly $f^{-1}(V)\cap M=(f|_M)^{-1}(V)$ and hence $f^{-1}(V)$ is open in $X$. It follows that $f$ is continuous.\\
%
%\noindent$(2).$ Forward implication follows from Lemma~\ref{LL2}.
%For the converse part we give only proof for the case when $f_M$ is closed. Let $F$ be closed in $X$. Let $M$ be a Menger subspace of $Y$. Clearly $f(F)\cap S=f_S(F\cap f^{-1}(M))$ is a closed subset of $M$ as $f_M$ is closed. It follows that $f(F)$ is a closed subset of $Y$ and hence $f$ is closed.
%\end{proof}
%

\begin{Prop}
\label{PP7}
 Let $Y$ be a MG-space and $f:X\to Y$ be a continuous mapping. If for any Menger subspace $M$ of $Y$ the restriction $f_M:f^{-1}(M)\to M$ is a quotient mapping, then $f:X\to Y$ is a quotient mapping.
\end{Prop}
%\begin{proof}
%Since $f_M$ is surjective for every Menger subspace $M$ of $Y$, it follows that $f$ is also surjective. Let $V\subseteq Y$ be such that $f^{-1}(V)$ is  open in $X$. Clearly $f^{-1}_M(V\cap M)=f^{-1}(V)\cap f^{-1}(M)$ is open in $f^{-1}(M)$ and hence $V\cap M$ is  open in $M$ for each Menger subspace $M$ of $Y$. Thus $V$ is an open subset of $Y$ and hence $f$ is a quotient mapping.
%\end{proof}

In combination with Lemma~\ref{LL2}, we have the following.
\begin{Cor}
\label{CC2}
If $Y$ is a Hausdorff $P$, MG-space, then a continuous mapping $f:X\to Y$ is a quotient mapping if and only if for any Menger subspace $M$ of $Y$ the restriction $f_M:f^{-1}(M)\to M$ is a quotient mapping.
\end{Cor}

%Similar to Proposition~\ref{PL1}, we obtain the following.
\begin{Prop}
\label{PP10}
%Let $\Lambda$ be an index set.
The topological sum $\oplus_{\alpha\in\Lambda}X_\alpha$ is a MG-space if and only if each $X_\alpha$ is a MG-space.
\end{Prop}
\begin{proof}
We give only proof for the converse part. Let $V$ be a subset of $\oplus_{\alpha\in\Lambda}X_\alpha$ such that for each Menger subspace $M$ of $\oplus_{\alpha\in\Lambda}X_\alpha$, $V\cap M$ is open in $M$. For each $\alpha\in\Lambda$, choose $U_\alpha=\{x\in X_\alpha : (x,\alpha)\in V\}$. Let $Y$ be a Menger subspace of $X_\alpha$. Clearly $V\cap(Y\times\{\alpha\})$ is open in $Y\times\{\alpha\}$. Then $U_\alpha$ is the image of $V\cap(Y\times\{\alpha\})$ under the projection mapping $p:Y\times\{\alpha\}\to Y$. Thus $U_\alpha$ is open in $Y$ and consequently $U_\alpha$ is open in $X_\alpha$. It follows that $U_\alpha\times\{\alpha\}$ is open in $X_\alpha\times\{\alpha\}$ and hence it is open in $\oplus_{\alpha\in\Lambda}X_\alpha$. Clearly $V=\cup_{\alpha\in\Lambda}(U_\alpha\times\{\alpha\})$ and thus we have $V$ is open in $\oplus_{\alpha\in\Lambda}X_\alpha$.
\end{proof}

\begin{Prop}
\label{PN2}
Suppose that $f:X\to Y$ is a mapping from a space $X$ onto a MG-space $Y$. If $f^{-1}(Z)$ is Menger for each Hausdorff Menger $P$-space $Z\subseteq Y$, then $f$ is closed.
\end{Prop}
\begin{proof}
Take a closed subset $F$ of $X$ and a Hausdorff Menger $P$-space $Z\subseteq Y$. Now $f^{-1}(Z)\cap F$ is Menger and $Z\cap f(F)=f(f^{-1}(Z)\cap F)$ is Menger as well. This shows that $Z\cap f(F)$ is closed in $Z$. Since $Y$ is a MG-space, $f(F)$ is closed in $Y$ and hence $f$ is closed.
\end{proof}

\begin{Prop}
\label{PP11}
If $X$ is a locally compact Hausdorff space and $Y$ is a MG-space, then $X\times Y$ is a MG-space.
\end{Prop}
\begin{proof}
Since $Y$ is a MG-space, there exist a locally Menger space $Z$ and a quotient mapping $q:Z\to Y$. Define $f:X\times Z\to X\times Y$ as $f(x,z)=(x,q(z))$.  Proceeding similarly as in the proof of Theorem~\ref{TT2} and using the fact that the projection $p:A\times B\to B$ is closed whenever $A$ is a compact, we can show that $f$ is a quotient mapping. It now follows that $X\times Y$ is a MG-space.
\end{proof}

\subsection{Decomposition of locally Menger spaces}
A collection $\mathcal{D}$ of subsets of a space $X$ is called a decomposition of $X$ if $\mathcal D$ forms a partition of $X$.
%provided that the members of $\mathcal A$ are pairwise disjoint and $X=\cup\mathcal A$.
The function $\varphi:X\to\mathcal{D}$ defined by $\varphi(x)=D_x$, where $x\in D_x$, is called the decomposition mapping. The topology on $\mathcal{D}$ can be defined as follows. A subset $\mathcal{U}$ of $\mathcal{D}$ is open in $\mathcal{D}$ if and only if $\varphi^{-1}(\mathcal{U})$ is open in $X$. In this case we say that $\mathcal D$ is a decomposition space (or in short, a decomposition) of $X$ \cite{Willard}. The decomposition mapping $\varphi:X\to\mathcal{D}$ is clearly surjective and continuous.
 The natural decomposition in $X$ of a mapping $f:X\to Y$ is the collection of disjoint sets $f^{-1}(y)$, $y\in Y$.
%\begin{Def}
%We say that  a space $X$ belongs to the class $\mathfrak{M}$ if there exists a cover $\mathcal{Y}$ of $X$ consisting of Menger subspaces and a subset $U$ of $X$ is open if $U\cap Y$ is open in $Y$ for each $Y\in\mathcal{Y}$ (or, equivalently a subset $F$ of $X$ is closed if $F\cap Y$ is closed in $Y$ for each $Y\in\mathcal{Y}$).
%\end{Def}
The collection of all MG-spaces is denoted by $\mathfrak{M}$.
%\begin{Prop}
%\label{T19}
%Every locally Menger space belongs to the class $\mathfrak{M}$.
%\end{Prop}
%\begin{proof}
%Let $X$ be  locally Menger. For each $x\in X$ choose an open subset $U_x$ and a Menger subspace $Y_x$ of $X$ such that $x\in U_x\subseteq Y_x$. The cover
%$\mathcal{Y}=\{Y_x : x\in X\}$ of $X$ witnesses that $X$ belongs to the class $\mathfrak{M}$.
%\end{proof}

\begin{Lemma}
\label{TL2}
Let $\mathcal{D}$ be a decomposition of $X$. If $X$ belongs to the class $\mathfrak{M}$, then $\mathcal{D}$ also belongs to the class $\mathfrak{M}$.
\end{Lemma}
\begin{proof}
Let $\mathcal{S}$ be a cover of $X$ consisting of Menger subspaces of $X$ and $\varphi:X\to\mathcal{D}$ be the decomposition mapping. Then $\varphi(\mathcal{S})=\{ \varphi(Y): Y\in\mathcal{S}\}$ is a cover of $\mathcal{D}$ consisting of Menger subspaces.
%Let $\mathcal{U}\subseteq \mathcal{D}$ be such that $\mathcal{U}\cap\varphi(Y)$ is open in $\varphi(Y)$ for each $Y\in\mathcal{S}$.
%%The proof will be complete if we show that $\mathcal{U}$ is open in $\mathcal D$.
%Let $Y\in\mathcal{S}$. Since $\varphi^{-1}(\mathcal{U}\cap\varphi(Y))
%=\varphi^{-1}(\mathcal{U})\cap \varphi^{-1}(\varphi(Y))$, it follows that $\varphi^{-1}(\mathcal{U})\cap Y=\varphi^{-1}(\mathcal{U}\cap\varphi(Y))\cap Y$. It is easy to see that $\varphi^{-1}(\mathcal{U}\cap\varphi(Y))$ is open in $\varphi^{-1}(\varphi(Y))$. Also, since $Y\subseteq \varphi^{-1}(\varphi(Y))$, we have $\varphi^{-1}(\mathcal{U})\cap Y$ is open in $Y$. Clearly, $\varphi^{-1}(\mathcal{U})$ is open in $X$ as the choice of $Y\in\mathcal{S}$ is arbitrary. Thus $\mathcal{U}$ is open in $\mathcal{D}$.
If $\mathcal{U}\subseteq \mathcal{D}$ is such that $\mathcal{U}\cap\varphi(Y)$ is open in $\varphi(Y)$ for each $Y\in\mathcal{S}$, then it is easy to verify that $\varphi^{-1}(\mathcal{U})\cap Y$ is open in $Y$ for each such $Y$. Accordingly $\varphi^{-1}(\mathcal{U})$ is open in $X$ and $\mathcal{U}$ is open in $\mathcal{D}$ as well.
\end{proof}

\begin{Th}
\label{TL3}
Every member of the class $\mathfrak{M}$ is homeomorphic to a decomposition of some locally Menger space.
\end{Th}
\begin{proof}
Let $X$ be a MG-space and consider a cover $\{X_\alpha : \alpha\in\Lambda\}$ of $X$. The space $Y=\cup_{\alpha\in\Lambda}Y_\alpha$ is locally Menger, where $Y_\alpha$'s are pairwise disjoint open Menger subspaces of $Y$ such that for each $\alpha$ there is a homeomorphism $h_\alpha:Y_\alpha\to X_\alpha$.
%The existence of such a space $Y$ is clear.
The mapping $f:Y\to X$ given by $f(y)=h_\alpha(y)$ for $y\in Y_\alpha$ is  continuous and surjective. Also if $f^{-1}(V)$ is open in $Y$, then $V$ is also open in $X$ as $X\in\mathfrak{M}$.
%Now $f^{-1}(V)\cap Y_\alpha$ is open in $Y_\alpha$ for each $\alpha$. Consequently, $h_\alpha(f^{-1}(V)\cap Y_\alpha)$ is open in $X_\alpha$ for each $\alpha$. It follows that $f(f^{-1}(V)\cap Y_\alpha)=h_\alpha(f^{-1}(V)\cap Y_\alpha)=V\cap X_\alpha$ is open in $X_\alpha$ for each $\alpha$. Since $X$ belongs to the class $\mathcal{C}_\beth$, $V$ is open in $X$.
Thus $\mathcal{D}=\{f^{-1}(x) : x\in X\}$ is a decomposition of $Y$. The mapping that sends each element $x$ of $X$ to $f^{-1}(x)$ is clearly a homeomorphism.
\end{proof}

%From the above results we obtain the following corollary.
\begin{Cor}
\label{CL1}
$\mathfrak{M}$ is identical with the class of all spaces which are obtained as a decomposition of locally Menger spaces.
\end{Cor}

 Another characterization of a locally Menger space can be obtained using bi-quotient mappings. Recall that bi-quotient image of a second countable space is again second countable.
\begin{Th}
\label{TD2}
Let $X$ be regular. Then $X$ is both locally Menger and locally metrizable if and only if $X$ is obtained as the image of a locally Menger metrizable space under a bi-quotient mapping.
%A regular space $X$ is obtained as the image of a locally Menger metrizable space under a bi-quotient mapping if and only if $X$ is locally Menger and locally metrizable.
\end{Th}
\begin{proof}
Let $X$ be locally Menger and locally metrizable. By Theorem~\ref{T7}, $X$ has a basis consisting of closed Menger neighbourhoods. Thus there is a cover $\{X_\alpha : \alpha\in\Lambda\}$ of $X$ such that $X_\alpha$ is a closed metrizable Menger subspace of $X$ for each $\alpha$. In line of the proof of Theorem~\ref{TL3}, we define a space $Y$ as follows.
\begin{enumerate}
  \item $Y=\cup_{\alpha\in\Lambda}Y_\alpha$, where each $Y_\alpha$ is Menger and metrizable.
  \item  $Y_\alpha$ is open in $Y$ for each $\alpha$ and $Y_\alpha\cap Y_\beta=\emptyset$ for $\alpha\neq\beta$.
  \item For each $\alpha$ there is a homeomorphism $h_\alpha:Y_\alpha\to X_\alpha$.
\end{enumerate}
Such $Y$ is  metrizable and also locally Menger. The mapping $f:Y\to X$ defined by $f(y)=h_\alpha(y)$ for $y\in Y_\alpha$ is open, continuous and surjective. Thus $f$ is a  bi-quotient mapping.
%Define $f:Y\to X$ by $f(y)=h_\alpha(y)$ for $y\in Y_\alpha$. Then $f$ is a continuous surjective mapping. We now show that $f$ is open. Let $V$ be an open subset of $Y$. Then for each $\alpha$, $V\cap Y_\alpha$ is open in $Y_\alpha$. This implies that for each $\alpha$, $f(V\cap Y_\alpha)=h_\alpha(V\cap Y_\alpha)$ is an open subset of $X_\alpha$ and hence $f(V\cap Y_\alpha)$ is an open subset of $X$. Thus we have $f$ is open and consequently $f$ is a bi-quotient mapping. Hence the result.

Conversely assume that $g:Y\to X$ is a bi-quotient mapping from a locally Menger metrizable space $Y$ onto the regular space $X$. By Theorem~\ref{T15}, $X$ is locally Menger. It now remains to show that $X$ is locally metrizable. Choose an open cover $\mathcal U=\{U_y : y\in Y\}$ of $Y$ such that $y\in U_y\subseteq Z_y$, where $Z_y$ is  Menger for each $y$. Let $x\in X$ and choose  $y\in g^{-1}(x)$.
%and then $\{V_y : y\in Y\}$ is a cover of $g^{-1}(x)$ by open sets in $Y$ and
 By definition of $g$, there exist a finite subset $\{U_{y_i} : 1\leq i\leq k\}$ of $\mathcal U$ and an open set $U$ in $X$ containing $x$ such that $U\subseteq\cup_{i=1}^k g(U_{y_i})$. Now $Z=\cup_{i=1}^k Z_{y_i}$ is a metrizable Menger subspace of $Y$ and hence $Z$ is second countable. Clearly $g(Z)$ is a metrizable subspace of $X$, as $g(Z)$ is a second countable regular  Menger space. Thus $U$ is the metrizable subspace of $X$, as required.
\end{proof}

\begin{Cor}
\label{CD1}
A regular space $X$ is obtained as the image of a locally Menger metrizable space under an open continuous mapping if and only if $X$ is locally Menger and locally metrizable.
\end{Cor}

\begin{Th}
\label{T2}
A Hausdorff $P$-space $X$ is obtained as the image of a locally Menger metrizable space under a bi-quotient mapping if and only if $X$ is locally Menger and locally metrizable.
\end{Th}

\begin{Cor}
\label{C1}
A Hausdorff $P$-space $X$ is obtained as the image of a locally Menger metrizable space under an open continuous mapping if and only if $X$ is locally Menger and locally metrizable.
\end{Cor}

\section{Remainders and associated observations}
Throughout the section all spaces are assumed to be Tychonoff. By a remainder of a space $X$ we mean the subspace $bX\setminus X$ of a compactification $bX$ of $X$. A space $X$ is said to be a $p$-space \cite{p-space,EGT} if in any (in some) compactification $bX$ of $X$ there exists a countable family $\{\mathcal{U}_n : n\in\mathbb{N}\}$ with for each $n$ $\mathcal{U}_n$ is a collection of open sets in $bX$ such that for each $x\in X$, $x\in\cap_{n\in\mathbb{N}}\cup\{U\in\mathcal{U}_n : x\in U\}\subseteq X$. It is known that every metrizable space is a $p$-space \cite{p-space1,AVR}. Also observe that the property of being a $p$-space is hereditary for closed subspaces. In \cite{Nagami}, continuous images of Lindel\"{o}f $p$-spaces are called as Lindel\"{o}f $\Sigma$-spaces.
 A space $X$ is said to be an $s$-space if there exists a countable open source for $X$ in any (in some) compactification $bX$ of $X$ \cite{s-space,s-space1}. It is well known that every Lindel\"{o}f $p$-space is an $s$-space \cite{s-space1}. Also if $X$ is a Lindel\"{o}f $p$-space, then any remainder of $X$ is also a Lindel\"{o}f $p$-space \cite[Theorem 2.1]{AVR}.

This section is devoted to study remainders of locally Menger spaces.
%We start with a folklore result about $p$-spaces.
%\begin{Prop}[Folklore]
%\label{TR1}
%The property of being a $p$-space is hereditary for closed subspaces.
%\end{Prop}
%\begin{proof}
%Let $F$ be a closed set in a $p$-space $X$. Choose a compactification $bX$ of $X$ and a countable family $\{\mathcal{U}_n : n\in\mathbb{N}\}$ that satisfies the condition of  $p$-spaces.
%%with for each $n$ $\mathcal{U}_n$ is a collection of open subsets of $bX$ such that for each $x\in X$, $x\in\cap\{St(x,\mathcal{U}_n) : n\in\mathbb{N}\}\subseteq X$.
%Since $Y=\overline{F}^{bX}$ is compact and  $F$ is dense in $Y$, it follows that $Y$ is a compactification of $F$. For each $n$ let $\mathcal{V}_n=\{U\cap Y : U\in\mathcal{U}_n\}$. The collection $\{\mathcal{V}_n : n\in\mathbb{N}\}$ is  countable, where each $\mathcal{V}_n$ consists of open subsets of $Y$. Let $y\in F$. We now show that $y\in\cap\{St(y,\mathcal{V}_n) : n\in\mathbb{N}\}\subseteq F$. Clearly, $y\in St(y,\mathcal{U}_n)$ for each $n$. Thus for each $n$ there exists a $U_n\in\mathcal{U}_n$ such that $y\in U_n$, i.e. $y\in U_n\cap Y\in\mathcal{V}_n$. Thus $y\in\cap\{St(y,\mathcal{V}_n) : n\in\mathbb{N}\}$. Observe that $St(y,\mathcal{V}_n)=St(y,\mathcal{U}_n)\cap Y$ holds for each $n$ and also by the given condition $\cap\{St(y,\mathcal{U}_n) : n\in\mathbb{N}\}\cap Y\subseteq X\cap Y$, i.e. $\cap\{St(y,\mathcal{V}_n) : n\in\mathbb{N}\}\subseteq X\cap Y$. Since $F=\overline{F}^{X}=Y\cap X$,  we can conclude that $y\in\cap\{St(y,\mathcal{V}_n) : n\in\mathbb{N}\}\subseteq F$. Hence $F$ is a $p$-space.
%\end{proof}
For a subspace $Y$ of a space $X$, we say that $X$ has the property $\mathcal{P}$ outside of $Y$ if each closed set $F$ in $X$ with $Y\cap F=\emptyset$ has the property $\mathcal{P}$.

\begin{Th}
\label{TR3}
Let $Y$ be a remainder of a regular locally Menger $p$-space $X$. Then there is a compact subspace $K$ of $Y$ such that $Y$ is a Lindel\"{o}f $p$-space outside of $K$ (and hence $Y$ is an $s$-space outside of $K$).
\end{Th}

\begin{proof}
Let $Y=bX\setminus X$, where $bX$ is a compactification of $X$. Choose an open cover $\mathcal{U}$ of $X$ such that $\overline{U}^X$ is Menger for each $U\in\mathcal{U}$. For each $U\in\mathcal{U}$ choose an open set $V_U$ in $bX$ such that $V_U\cap X=U$. Consider $W=\cup\{V_U : U\in\mathcal{U}\}$ and $K=bX\setminus W$. Now $K$ is compact and $W$ is open in $bX$ with $X\subseteq W$ and $K\subseteq Y$. It remains to show that $Y$ is a Lindel\"{o}f $p$-space outside of $K$. Let $F$ be a closed set in $Y$ such that $K\cap F=\emptyset$. Then $\overline{F}^{bX}\subseteq W$ as $K\cap\overline{F}^{bX}=\emptyset$. Since $\{V_U : U\in\mathcal{U}\}$ is a cover of $\overline{F}^{bX}$ by open sets in $bX$ and $\overline{F}^{bX}$ is  compact, there is a finite set $\{V_{U_i} : 1\leq i\leq k\}$ such that $\overline{F}^{bX}\subseteq\cup_{i=1}^k V_{U_i}$. Now $C=\cup_{i=1}^k\overline{U}_i^X$ is a closed Menger subspace of $X$ and hence $C$ is a $p$-space. Also $Z=\overline{C}^{bX}$ is compact and $C$ is dense in $Z$. Clearly $Z$ is a compactification of $C$ and $Z\cap Y$ is the remainder of $C$ in $Z$. Thus $Z\cap Y$ is a Lindel\"{o}f $p$-space as it is the remainder of a Lindel\"of $p$-space. Since $\cup_{i=1}^k V_{U_i}\cap\overline{X}^{bX}\subseteq\overline{\cup_{i=1}^k V_{U_i}\cap X}^{bX}$, we have $\cup_{i=1}^k V_{U_i}\subseteq\overline{\cup_{i=1}^k U_i}^{bX}\subseteq\overline{C}^{bX}=Z$ and hence $\overline{F}^{bX}\subseteq Z$. Also since $F=\overline{F}^Y=\overline{F}^{bX}\cap Y$, we have $F\subseteq Z\cap Y$. Evidently $F$ is closed in $Z\cap Y$ and hence $F$ is a Lindel\"{o}f $p$-space. This completes the proof.
\end{proof}

\begin{Cor}
\label{CR1}
Let $X$ be a locally Menger Hausdorff $P$-space. If $X$ is a $p$-space and $Y$ is a remainder of $X$, then there exists a compact subspace $K$ of $Y$ such that $Y$ is a Lindel\"{o}f $p$-space outside of $K$ (and hence $Y$ is an $s$-space outside of $K$).
\end{Cor}

Recall that a space $X$ is said to be homogeneous if for any $x,y\in X$ there exists a homeomorphism $f:X\to X$ such that $f(x)=y$.

It is known that if a space $X$ is the union of a finite collection of its closed $s$-subspaces, then $X$ is an $s$-space (see \cite{Wang}). This observation will be used in the next result.
\begin{Th}
\label{TR4}
If a locally Menger $p$-space $X$ has a homogeneous remainder $Y$, then $Y$ is an $s$-space.
\end{Th}
\begin{proof}
By Theorem~\ref{TR3}, there is a compact subspace $K$ of $Y$ such that $Y$ is a Lindel\"{o}f $p$-space outside of $K$. Now the proof will follow if $Y=K$.
Otherwise for each $y\in Y\setminus K$ there exists an open set $U_y$ in $Y$ such that $y\in U_y\subseteq\overline{U_y}^Y\subseteq Y\setminus K$ as $Y\setminus K$ is open in $Y$. Clearly every such $\overline{U_y}^Y$ is a Lindel\"{o}f $p$-space.
Next we show that for each $x\in Y$ there is an open set $V$ in $Y$ containing $x$ such that $\overline{V}^Y$ is a Lindel\"{o}f $p$-space. Let $x\in Y$ and fix a $y\in Y\setminus K$. Choose a homeomorphism $f:Y\to Y$ with $f(y)=x$. Next choose an open set $U_y$ in $Y$ containing $y$ such that $\overline{U_y}^Y$ is a Lindel\"{o}f $p$-space.
Now $V=f(U_y)$ is the required open set in $Y$.
We thus obtain a cover $\mathcal{W}$ of $K$ by open sets in $Y$ such that for each $W\in\mathcal{W}$, $\overline{W}^Y$ is an $s$-space. Choose a finite set $\{W_i : 1\leq i\leq k\}\subseteq \mathcal{W}$ such that $K\subseteq\cup_{i=1}^k W_i$. Clearly $Y\setminus\cup_{i=1}^k W_i$ is an $s$-space.  Therefore $Y=\cup_{i=1}^k\overline{W_i}^Y\cup (Y\setminus\cup_{i=1}^k W_i)$ is also an $s$-space.
\end{proof}

\begin{Cor}
\label{CR2}
Let $X$ be a locally Menger Hausdorff $P$-space. If in addition $X$ is a $p$-space with a homogeneous remainder $Y$, then $Y$ is an $s$-space.
\end{Cor}

 By \cite[Theorem 2.7]{AGCL}, any (some) remainder of an $s$-space $X$ in a compactification of $X$ is a Lindel\"{o}f $\Sigma$-space. We utilize this result to obtain the following.
%For the next result we need the following observation about $s$-spaces from \cite{AGCL}.
%\begin{Th}[{\cite[Theorem 2.7]{AGCL}}]
%\label{TR5}
%If $X$ is an $s$-space, then any (some) remainder of $X$ in a compactification of $X$ is a Lindel\"{o}f $\Sigma$-space.
%\end{Th}

\begin{Th}
\label{TR6}
Let $X$ be a regular locally Menger $p$-space. If $X$ has a homogeneous remainder, then $X=L\cup Z$, where $L$ is a closed Lindel\"{o}f $\Sigma$-subspace and $Z$ is an open locally compact subspace of $X$.
\end{Th}
\begin{proof}
Let $bX$ be a compactification of $X$. If $Y=bX\setminus X$ is a homogeneous remainder of $X$, then $Y$ is an $s$-space by Theorem~\ref{TR4}. Now $bY=\overline{Y}^{bX}$ is a compactification of $Y$. Next choose $L=bY\cap X$. Then $L$ is closed in $X$ and $L=bY\setminus Y$. Clearly $L$ is the remainder of $Y$ in $bY$. Also $L$ is a Lindel\"{o}f $\Sigma$-subspace of $X$ by \cite[Theorem 2.7]{AGCL}. For each $x\in Z$, by the openness of $Z=bX\setminus bY$ in $bX$, choose an open set $U_x$ in $bX$ such that $x\in U_x\subseteq\overline{U_x}^{bX}\subseteq Z$. Clearly $Z$ is locally compact and  $X=L\cup Z$.
\end{proof}

\begin{Cor}
\label{PR1}
Let $X$ be a regular locally Menger $p$-space. If $X$ has a homogeneous remainder and in addition $X$ is nowhere locally compact, then $X$  is a Lindel\"{o}f $\Sigma$-space.
\end{Cor}
%\begin{proof}
%We have $X=L\cup Z$, where $L$ is a closed Lindel\"{o}f $\Sigma$-subspace and $Z$ is an open locally compact subspace of $X$. Clearly, $Z=\emptyset$ as $X$ is nowhere locally compact.
%\end{proof}

\begin{Cor}
\label{CR3}
Let $X$ be a locally Menger Hausdorff $P$-space. If $X$ is a $p$-space with a homogeneous remainder, then $X=L\cup Z$, where $L$ is a closed Lindel\"{o}f $\Sigma$-subspace and $Z$ is an open locally compact subspace of $X$.
\end{Cor}

\begin{Cor}
\label{C4}
Let $X$ be a locally Menger Hausdorff $P$-space. If $X$ is a $p$-space with a homogeneous remainder and in addition $X$ is nowhere locally compact, then $X$  is a Lindel\"{o}f $\Sigma$-space.
\end{Cor}

Let $\beta X$ denote the Stone-\v{C}ech compactification of $X$. For any property $\mathcal P$, we say that $X$ has property $\mathcal P$ at infinity if $\beta X\setminus X$ has property $\mathcal P$.

\begin{Def}[{\cite{SPC}}]
\label{DN1}
A continuous mapping $f$ of a space $X$ onto a space $Y$ is called a meshing map of $X$ onto $Y$ if there exist compactifications $bX$ of $X$ and $cY$ of $Y$, and a continuous extension $\tilde{f}$ of $f$ over $bX$ onto $cY$ such that $\tilde{f}|_{bX\setminus X}$ is a homeomorphism onto $cY\setminus Y$.
\end{Def}

Observe that every meshing map is perfect.

\begin{Th}[{\cite{Cech}}]
\label{TN1}
Any compactification $bX$ of a space $X$ is the image of $\beta X$ under a (unique) continuous mapping $f$ that keeps $X$ pointwise fixed. Furthermore $f(\beta X\setminus X)=bX\setminus X$.
\end{Th}

It is easy to observe that both $f$ and its restriction to $\beta X\setminus X$ in Theorem~\ref{TN1} are perfect mappings. Use Corollary~\ref{T1501} to obtain the following.
\begin{Cor}
\label{CN1}
A space $X$ is locally Menger at infinity if and only if for any compactification $bX$ of $X$, $bX\setminus X$ is locally Menger.
\end{Cor}

For any space $X$, let $N(X)$ denote the set of points of $X$ at which $X$ is not locally Menger. It is immediate that $N(X)$ is closed.

\begin{Prop}
\label{PN1}
For any compactification $bX$ of $X$, $N(X)\subseteq\overline{bX\setminus X}$.
\end{Prop}
\begin{proof}
Let $x\in N(X)$ and $U$ be an open set in $bX$ containing $x$. If possible suppose that $U\cap (bX\setminus X)=\emptyset$. Then we get an open set $V$ in $bX$ containing $x$ such that $\overline{V}\subseteq U$ and hence $\overline{V}\subseteq X$. Thus $V$ is an open set in $X$ with $\overline{V}$ is compact. It follows that $X$ is locally Menger at $x$ (even $X$ is locally compact at $x$), a contradiction. Thus $U\cap (bX\setminus X)\neq\emptyset$ and hence $x\in\overline{bX\setminus X}$. This completes the proof.
\end{proof}

We say that a closed continuous mapping $f$ from a space $X$ onto a space $Y$ is nearly perfect if for each $y\in Y$, $f^{-1}(y)$ is Menger. We now present few observations about nearly perfect mappings.

\begin{Lemma}[{\cite[Corollary 3.6]{Box}}]
\label{LN1}
If $f$ is a nearly perfect mapping from a space $X$ onto a Lindel\"{o}f $P$-space, then $X$ is a Menger space.
\end{Lemma}

\begin{Lemma}
\label{LN2}
Let $f$ be a continuous mapping from a space $X$ onto a space $Y$. The following assertions are equivalent.
\begin{enumerate}[wide,label={\upshape(\arabic*)},
leftmargin=*]
  \item The mapping $f$ is nearly perfect.

  \item\begin{enumerate}[wide=0pt,label={\upshape(\roman*)},leftmargin=*]
          \item For each $y\in Y$ and each open set $U$ in $X$ containing $f^{-1}(y)$, there exists an open set $V$ in $Y$ containing $y$ such that $f^{-1}(V)\subseteq U$.
          \item For each Lindel\"{o}f $P$-space $Z\subseteq Y$, $f^{-1}(Z)$ is Menger.
        \end{enumerate}
\end{enumerate}
\end{Lemma}
\begin{proof}
$(1)\Rightarrow (2)$. (i) Let $y\in Y$ and $U$ be an open set in $X$ with $f^{-1}(y)\subseteq U$. Then $V=Y\setminus f(X\setminus U)$ is an open set in $Y$ containing $y$ as $f$ is closed and $f^{-1}(y)\subseteq U$. Clearly $f^{-1}(V)\subseteq U$.

(ii) Let $Z$ be a Lindel\"{o}f $P$-space of $Y$. Then the restriction $f_Z:f^{-1}(Z)\to Z$ is closed (see Lemma~\ref{LL201}) and continuous. Obviously for each $z\in Z$, $f_Z^{-1}(z)$ is a Menger subspace of $f^{-1}(Z)$ and hence by Lemma~\ref{LN1}, $f^{-1}(Z)$ is also Menger.\\

$(2)\Rightarrow (1)$. We claim that $f$ is closed. Let $F$ be a closed subset of $X$ such that $F\neq X$. Let $y\in Y\setminus f(F)$. Then $X\setminus F$ is open in $X$ containing $f^{-1}(y)$. From the hypothesis (i), there exists an open set $V$ in $Y$ containing $y$ such that $f^{-1}(V)\subseteq X\setminus F$. It follows that $V\subseteq Y\setminus f(F)$ as for any $z\in V$ we have $f^{-1}(z)\cap F=\emptyset$ and so $z\in Y\setminus f(F)$. Thus $f$ is closed.

From the hypothesis (ii), it is clear that for each $y\in Y$, $f^{-1}(y)$ is Menger.
\end{proof}

\begin{Th}
\label{TN2}
If $f:X\to Y$ is a nearly perfect mapping from a space $X$ onto a $P$-space $Y$, then $f(N(X))=N(Y)$.
\end{Th}
\begin{proof}
Let $y\in f(N(X))$. Then $y=f(x)$ for some $x\in N(X)$. Suppose that $y\notin N(Y)$. Choose an open set $U$ and a Menger subspace $Z$ of $Y$ such that $y\in U\subseteq Z$. This gives us $f^{-1}(y)\subseteq f^{-1}(U)\subseteq f^{-1}(Z)$ i.e. $x\in f^{-1}(U)\subseteq f^{-1}(Z)$, where $f^{-1}(U)$ is open in $X$ and $f^{-1}(Z)$ is a Menger subspace of $X$. It follows that $X$ is locally Menger at $x$, which is a contradiction. Thus $y\in N(Y)$ and hence $f(N(X))\subseteq N(Y)$.

Conversely take $y\in N(Y)$. Suppose that $y\notin f(N(X))$. Clearly $f^{-1}(y)\cap N(X)=\emptyset$ i.e. $f^{-1}(y)\subseteq X\setminus N(X)$. For each $x\in f^{-1}(y)$ we choose an open set $U_x$ and a Menger subspace $M_x$ of $X$ such that $x\in U_x\subseteq M_x$. Now the collection $\{U_x : x\in f^{-1}(y)\}$ is a cover of $f^{-1}(y)$ by open sets in $X$. Since $f^{-1}(y)$ is Menger, there exists a countable subset $\{U_{x_i} : i\in\mathbb{N}\}$ such that $f^{-1}(y)\subseteq\cup_{i\in\mathbb{N}}U_{x_i}
\subseteq\cup_{i\in\mathbb{N}}M_{x_i}$. By Lemma~\ref{LN2}, we can find an open set $V$ in $Y$ containing $y$ such that $f^{-1}(V)\subseteq\cup_{i\in\mathbb{N}}U_{x_i}$. Subsequently $V\subseteq f(\cup_{i\in\mathbb{N}}U_{x_i})\subseteq f(\cup_{i\in\mathbb{N}}M_{x_i})$. Since $U_{i\in\mathbb{N}}M_{x_i}$ is Menger, $f(\cup_{i\in\mathbb{N}}M_{x_i})$ is also Menger and hence $y\notin N(Y)$, which is a contradiction. Thus $y\in f(N(X))$ and consequently $N(Y)\subseteq f(N(X))$. This completes the proof.
\end{proof}

%\begin{Lemma}[{\cite[Lemma 1.4(b)]{SPC}}]
%\label{LN3}
%If $f$ is a meshing mapping from a space $X$ onto $Y$ and $A$ is a closed subset of $X$, then $f|_A$ is meshing onto $f(A)$.
%\end{Lemma}

\begin{Lemma}[{\cite[Theorem 88]{Stone}}]
\label{LN4}
If $f$ is any continuous mapping from a space $X$ onto a space $Y$, and if $bY$ is any compactification of $Y$, then $f$ has a (unique) continuous extension $f_b$ over $\beta X$ onto $bY$.
\end{Lemma}

Recall that if $f$ is a continuous mapping of a space $X$ onto a space $Y$, then a cross section of $f$ is a continuous mapping $g$ of $Y$ into $X$ such that $f\circ g$ is the identity map of $Y$ onto itself.

\begin{Th}
\label{TN3}
If a meshing map $f$ of a space $X$ onto a $P$-space $Y$ has a cross section, then $f|_{N(X)}$ is a homeomorphism onto $N(Y)$.
\end{Th}
\begin{proof}
By Theorem~\ref{TN2}, $f|_{N(X)}$ is a closed continuous mapping onto $N(Y)$. We now show that $f|_{N(X)}$ is injective. Since $f$ is meshing, we get compactifications $bX$ of $X$ and $cY$ of $Y$ and an extension $\tilde{f}$ over $bX$ onto $cY$ such that $\tilde{f}|_{bX\setminus X}$ is a homeomorphism onto $cY\setminus Y$. Let $g$ be a cross section of $f$. Then by Lemma~\ref{LN4}, $g$ has a continuous extension $\tilde{g}$ over $\beta Y$ into $bX$. Since $f\circ g$ is the identity map of $Y$ onto itself, by Theorem~\ref{TN1}, we have $\tilde{f}\circ\tilde{g}(\beta Y\setminus Y)=cY\setminus Y$ and consequently $\tilde{f}\circ\tilde{g}$ maps $\beta Y$ onto $cY$. Now $cY\setminus Y=\tilde{f}(\tilde{g}(\beta Y\setminus Y))$ gives $bX\setminus X\subseteq\tilde{g}(\beta\setminus Y)\subseteq\tilde{g}(\beta Y)$ as $\tilde{f}|_{bX\setminus X}$ maps $bX\setminus X$ bijectively $cY\setminus Y$. Obviously $bX\setminus X\subseteq\tilde{g}(\beta Y)$ and hence $N(X)\subseteq\tilde{g}(\beta Y)\cap X$ as $N(X)\subseteq\overline{bX\setminus X}$.

Next we claim that $\tilde{g}(\beta Y)\cap X=g(Y)$. Again using $\tilde{f}\circ\tilde{g}(\beta Y\setminus Y)=cY\setminus Y$ we obtain $\tilde{g}(\beta Y\setminus Y)=bX\setminus X$ and so $\tilde{g}(\beta Y)\cap X=\tilde{g}((\beta Y\setminus Y)\cup Y)\cap X=(\tilde{g}(\beta Y\setminus Y)\cup\tilde{g}(Y))\cap X=(\tilde{g}(\beta Y\setminus Y)\cap X)\cup(\tilde{g}(Y)\cap X)=((bX\setminus X)\cap X)\cup(g(Y)\cap X)=g(Y)$.

Since $g$ is a cross section of $f$, $f$ is injective on $g(Y)$ and so on $N(X)$. Hence $f|_{N(X)}$ is a homeomorphism onto $N(Y)$.
\end{proof}

\begin{Prop}
\label{PN3}
If $X=K\times Y$ with $K$ is locally $\sigma$-compact, then $N(X)=K\times N(Y)$.
\end{Prop}

\begin{Prop}
\label{PN4}
Let $X$ be a non-compact Menger space containing at least two points and $Y$ be a $P$-space that is not locally Menger, then the projection mapping $p:X\times Y\to Y$ is a nearly perfect mapping but not perfect (hence not meshing).
\end{Prop}
\begin{proof}
By Lemma~\ref{LL1}, $p$ is closed. For each $y\in Y$, $p^{-1}(y)=X\times\{y\}$ is Menger. Thus $p$ is nearly perfect. Since $X$ is not compact, $p^{-1}(y)$ is not compact for all $y\in Y$ and hence $p$ is not perfect.
\end{proof}

\begin{Prop}
\label{PN5}
Let $X$ be a Menger space containing at least two points and $Y$ be a $P$-space that is not locally Menger, then the projection mapping $p:X\times Y\to Y$ is a nearly perfect mapping but not meshing.
\end{Prop}
\begin{proof}
Observe that $p$ is a nearly perfect mapping. We now show that $p$ is not meshing. Suppose that $p$ is meshing. Since every projection mapping has a cross section, by Theorem~\ref{TN3}, we can say that $p^\prime=p|_{N(X\times Y)}:N(X\times Y)\to N(Y)$ is a homeomorphism. Also it is easy to see that $X\times N(Y)\subseteq N(X\times Y)$. Thus $p^\prime|_{X\times N(Y)}:X\times N(Y)\to N(Y)$ is injective, which is a contradiction. Hence $p$ is not meshing.
\end{proof}

\begin{Th}
\label{TN4}
If $f$ is a nearly perfect mapping from a Hausdorff $P$-space $X$ onto a space $Y$ and $A$ is a subset of $X$, then $f|_A$ is a nearly perfect mapping onto $f(A)$ if and only if $A=f^{-1}(f(A))\cap\overline{A}$.
\end{Th}
\begin{proof}
Let $A$ be a dense subset of $X$. Suppose that $f|_A$ is nearly perfect onto $f(A)$. We now show that $A=f^{-1}(f(A))$. Let $x\in X\setminus A$ and $y\in f(A)$ be arbitrarily chosen. Then $f|_A^{-1}=f^{-1}(A)\cap A$ is Menger and hence closed. It follows that there exists an open set $U$ in $A$ such that $f|_A^{-1}(y)\subseteq U$ and $x\notin\overline{U}$. Subsequently $f(A\setminus U)$ is a closed subset of $f(A)$ not containing $y$. We claim that $x\in\overline{f(A\setminus U)}$. Let $V$ be an open set in $Y$ containing $f(x)$. Then $f^{-1}(V)\cap(X\setminus\overline{U})$ is an open set in $X$ containing $x$ and so $f^{-1}(V)\cap(X\setminus\overline{U})\cap A=\emptyset$. This gives $f^{-1}(V)\cap (A\setminus U)\neq\emptyset$ i.e. $V\cap f(A\setminus U)\neq\emptyset$ and consequently $f(x)\in\overline{f(A\setminus U)}$. It is clear that $y\notin\overline{f(A\setminus U)}$ as $f(A\setminus U)=\overline{f(A\setminus U)}^{f(A)}=\overline{f(A\setminus U)}\cap f(A)$. Thus $y\neq f(x)$ and hence $A=f^{-1}(f(A))$ i.e. $A=f^{-1}(f(A))\cap\overline{A}$.

Conversely suppose that $A=f^{-1}(f(A)$ i.e. $A=f^{-1}(f(A))$. Let $F$ be a closed subset of $A$. Then $F=C\cap A$ for some closed $C$ in $X$. We claim that $f|_A(F)=f(C)\cap f(A)$. Let $y\in f(C)\cap f(A)$. Since $A=f^{-1}(f(A))$, $f^{-1}(y)\subseteq A$. Then $f^{-1}(y)\subseteq A$ and $y\in f(C)$ give a $x\in C\cap A=F$ such that $y=f(x)$ i.e. $y\in f|_A(F)$. Thus $f|_A(F)=f(C)\cap f(A)$. Therefore $f|_A$ is closed and so $f|_A$ is nearly perfect onto $f(A)$. Hence the result holds for any dense subset of $X$.

For any closed subset $A$ of $X$, obviously $f|_A$ is nearly perfect onto $f(A)$, and $A=f^{-1}(f(A))\cap\overline{A}$. Finally let $A$ be any subset of $X$ and assume that $f|_A$ is nearly perfect onto $f(A)$. If we consider $f|_A$ as the restriction to $A$ of the nearly perfect mapping $f|_{\overline{A}}$, then we can obtain $A=f^{-1}(f(A))\cap\overline{A}$. Conversely let $A=f^{-1}(f(A))\cap\overline{A}$. Then $f|_{\overline{A}}$ is nearly perfect. Since $A$ is dense in $\overline{A}$, we get $f|_{\overline{A}}^{-1}(f|_{\overline{A}}(A))=f^{-1}(f(A))\cap\overline{A}=A$ and hence from the above $g|_A=f|_A$, where $g=f|_{\overline{A}}$, is nearly perfect onto $f(A)$. This completes the proof.
\end{proof}
Finally we obtain the following.
\begin{Th}
\label{TN5}
If $f$ is a meshing map from a space $X$ onto a space $Y$, then for any pair $x_1, x_2$ of distinct points of $N(X)$ such that $f(x_1)=f(x_2)$ there exist neighbourhoods $U_1$ of $x_1$ and $U_2$ of $x_2$ in $\beta X$ such that $f_\beta(U_1)\cap f_\beta(U_2)$ is a compact subset of $Y$.
\end{Th}
\begin{proof}
Since $f$ is a meshing, we obtain compactifications $bX$ of $X$ and $cY$ of $Y$ and an extension $\tilde{f}$ from $bX$ onto $cY$ which maps $bX\setminus X$ homeomorphically onto $cY\setminus Y$. By Theorem~\ref{TN1}, there exist mappings $g_b$ from $\beta X$ onto $bX$ and $g_c$ from $\beta Y$ onto $cY$ keeping $X$ and $Y$ pointwise fixed respectively and sending $\beta X\setminus X$ onto $bX\setminus X$ and $\beta Y\setminus Y$ onto $cY\setminus Y$ respectively. Since both $\tilde{f}\circ g_b$ and $g_c\circ f_\beta$ send $\beta X$ onto $cY$, and coincide with $f$ on $X$, by Lemma~\ref{LN4}, $\tilde{f}\circ g_b=g_c\circ f_\beta$. Let $V_1$ and $V_2$ be disjoint closed neighbourhoods in $bX$ of the distinct points $x_1$ and $x_2$ of $N(X)$ respectively. Choose $U_1=g_b^{-1}(V_1)$ and $U_2=g_b^{-1}(V_2)$. Then $U_1$ and $U_2$ are neighbourhoods of $x_1$ and $x_2$ in $\beta X$ respectively, and $(\tilde{f}\circ g_b (U_1))\cap(\tilde{f}\circ g_b(U_2))=\tilde{f}(V_1)\cap\tilde{f}(V_2)$ is compact. Since $\tilde{f}$ is injective on $bX\setminus X$, $(\tilde{f}\circ g_b (U_1))\cap(\tilde{f}\circ g_b(U_2))$ is a compact subset of $Y$. Clearly $(\tilde{f}\circ g_b (U_1))\cap(\tilde{f}\circ g_b(U_2))=(g_c\circ f_\beta(U_1))\cap(g_c\circ f_\beta(U_2))$ and since $g_c$ maps $\beta Y\setminus Y$ onto $cY\setminus Y$, it follows that $f_\beta(U_1)\cap f_\beta(U_2)$ is a compact subset of $Y$.
\end{proof}

\section{Concluding Remarks}
The investigation of the locally Menger property may further be continued to study separation properties. We first observe that if $L$ is a Lindel\"{o}f subspace of a Hausdorff $P$-space $X$ and $x_0\notin L$, then there exist two  open sets $U$ and $V$ such that $L\subseteq U$, $x_0\in V$ and $U\cap V=\emptyset$. A similar investigations may be carried out in the context of locally Menger spaces.
%\subsection{{Separation properties}}
%We start with a basic observation that a point and a  Lindel\"of subspace can be separated in a Hausdorff $P$-space.

%\begin{Prop}[Folklore]
%\label{L1}
%Let $X$ be a Hausdorff $P$-space and $L$ be a Lindel\"{o}f subspace of $X$ such that $x_0\notin L$. Then there exist two  open sets $U$ and $V$ such that $L\subseteq U$, $x_0\in V$ and $U\cap V=\emptyset$.
%\end{Prop}
%\begin{proof}
%Let $p\in L$. We can find disjoint open sets $U_p$ and $V_p$ containing $p$ and $x_0$ respectively. Now the cover $\{U_p\}_{p\in L}$ of $L$ by open sets in $X$ has a countable subcollection $\{U_p\}_{p\in I}$ that covers $L$, where $I$ is a countable subset of $L$. Clearly $U=\cup_{p\in I}U_p$ and $V=\cap_{p\in I}V_p$ are the required disjoint open sets such that $x_0\in V$ and $L\subseteq U$.
%\end{proof}

A point and a Lindel\"{o}f subspace can be separated by Menger spaces.
\begin{Prop}
\label{T18}
Let $X$ be a regular locally Menger  space. For each $x$ and each Lindel\"{o}f subspace $L$ of $X$ not containing $x$, there exist a closed Menger subspace $Y$ and a Menger subspace $Z$ such that $x\in Y$, $L\subseteq Z$ and $Y\cap Z=\emptyset$.
\end{Prop}
%\begin{proof}
%For each $y\in L$ choose disjoint open sets $U_y$ and $V_y$ containing $x$ and $ y$ respectively. By Theorem~\ref{T7}, we get two open sets $O_y,W_y$ in $X$ such that $x\in O_y\subseteq \overline{O}_y\subseteq U_y$ and $y\in W_y\subseteq\overline{W}_y\subseteq V_y$ with $\overline{O}_y$ and $\overline{W}_y$ are Menger subspaces of $X$. Clearly $\{W_y : y\in L\}$ is a cover of $L$ by open sets in $X$ and hence there is a countable subfamily $\{W_{y_n} : n\in\mathbb{N}\}$ that covers $L$. Thus $Y=\cap_{n\in\mathbb{N}}\overline{O}_{y_n}$ is a  closed Menger and $Z=\cup_{n\in\mathbb{N}}\overline{W}_{y_n}$ is a Menger subspace of $X$ such that $x\in Y$, $L\subseteq Z$ and $Y\cap Z=\emptyset$.
%\end{proof}

A similar situation can be described for locally Menger Hausdorff $P$-spaces.
\begin{Prop}
\label{T1}
Let $X$ be a locally Menger Hausdorff $P$-space. For each $x$ and each Lindel\"{o}f subset $L$ of $X$ not containing $x$, there exist two closed Menger subspaces $Y$ and $Z$ such that $x\in Y,\; L\subseteq Z$ and $Y\cap Z=\emptyset$.
\end{Prop}

Similarly a point and a compact subspace can be separated in a regular locally Menger spaces.
\begin{Prop}
\label{T19}
Let $X$ be a regular locally Menger space. For each $x$ and each compact subset $K$ of $X$ not containing $x$, there exist two closed Menger subspaces $Y$ and $Z$ such that $x\in Y,\; K\subseteq Z$ and $Y\cap Z=\emptyset$.
\end{Prop}

The separation by a continuous function in locally Menger spaces can also be obtained.
\begin{Prop}
\label{T17}
Let $X$ be a regular locally Menger space. For each compact set $K$ and each open set $V$ containing $K$, there exists a closed Menger subspace $Y$ such that $K\subseteq Y\subseteq V$. Furthermore, there is a continuous function $f:X\to [0,1]$ satisfying $f(x)=0$ for all $x\in K$ and $f(x)=1$ for all $x\in X\setminus Y$.
\end{Prop}
\begin{proof}
Let $x\in K$. We can find an open set $U_x$ such that $x\in U_x\subseteq\overline{U}_x\subseteq V$. Then $\{U_x : x\in K\}$ is a cover of $K$ by open sets in $X$. Choose a finite subfamily $\{U_{x_1}, U_{x_2}, \ldots, U_{x_n}\}$ that covers $K$. Clearly $Y=\cup_{k=1}^n\overline{U}_{x_k}$ is a closed Menger (and hence normal) subspace of $X$ and $K\subseteq Y\subseteq V$.

Since $Y\setminus \Int(Y)$ and $K$ are disjoint closed subsets of $Y$, there is a continuous function $g:Y\to [0,1]$ such that $g(x)=0$ for all $x\in K$ and $g(x)=1$ for all $x\in Y\setminus \Int(Y)$.

Now the required continuous map $f:X\to [0,1]$ is given by
\begin{equation*}
f(x)=
\begin{cases}
g(x) & \text{if~} x\in Y\\
1 & \text{otherwise}.
\end{cases}
\end{equation*}
\end{proof}

Likewise the following result can be verified.
\begin{Prop}
Let $X$ be a locally Menger Hausdorff $P$-space. Given a Lindel\"{o}f subspace $L$ and an open set $V$ containing $L$, there exists a closed Menger subspace $Y$ such that $L\subseteq Y\subseteq V$. Furthermore, there is a continuous function $f:X\to [0,1]$ such that $f(x)=0$ for all $x\in L$ and $f(x)=1$ for all $x\in X\setminus Y$.
\end{Prop}

Another line of investigations is the following.
%Next we define Menger-covering mapping.
\begin{Def}
\label{D4}
A continuous mapping $f:X\rightarrow Y$ is said to be a Menger-covering mapping if for each Menger subspace $B$ of $Y$ there exists a Menger subspace $A$ of $X$ such that $f(A)=B$.
\end{Def}
For example, the projection mappings $p_1:X\times Y\to X$ and $p_2:X\times Y\to Y$ are Menger-covering mappings. Also it is easy to verify that Menger property is an inverse invariant under injective closed as well as injective Menger-covering mappings. Also it can be observed that the locally Menger property is an inverse invariant under injective Menger-covering mappings.
Note that the Menger (respectively, locally Menger) property is not an inverse invariant under Menger-covering mappings.
 %Take $X=X_1\times X_2$, where $X_1=[0,\omega_1)$, $X_2=[0,\omega_1]$. Consider the projection $p_2:X\to X_2$. Clearly $p_2$ is a Menger-covering open mapping and $X_2$ is Menger whereas $X$ is not Menger as $X_1$ is not.
%Also consider $X=X_1\times X_2$, where $X_1=\Psi(\mathcal{A})$ the Isbell-Mr\'{o}wka space and $X_2=\mathbb{R}^2$ with radial interval topology. Let $p_1:X\to X_1$ be the projection onto first coordinate. Clearly $p_1$ is a Menger-covering open mapping and $X_1$ is locally Menger (see Example~\ref{E1}). Now $X$ is not locally Menger as $X_2$ is not (see Example~\ref{P10}).
The behaviour of MG-spaces under the Menger-covering mappings seems to be an interesting line of investigations.

{}
\end{document}